\documentclass[11pt, a4paper, reqno]{amsart}

\usepackage{tikz-cd}
\usepackage{array}
\usetikzlibrary{snakes}

\usepackage{amsmath, amssymb, amsfonts, amstext, amsthm, amscd}
\usepackage[mathscr]{euscript}
\usepackage{enumerate}

\usepackage{float}
\usepackage{graphicx} 

\usepackage[paperheight=11.2in,paperwidth=8.1in,bindingoffset=0in,left=0.9in,right=0.9in,top=1in,bottom=1in,headsep=.5\baselineskip]{geometry}
\usepackage{tikz}
\usepackage{tikz-cd}
\usetikzlibrary{snakes}
\usetikzlibrary{intersections, calc}
\usetikzlibrary{decorations.markings}
\usetikzlibrary{arrows.meta}
\usetikzlibrary{bending}

\usepackage[english]{babel}
\usepackage{chngcntr}

\begin{document}
	
	\thispagestyle{empty} 
	\def\theequation{\arabic{section}.\arabic{equation}}

	\newcommand{\codim}{\mbox{{\rm codim}$\,$}}
	\newcommand{\stab}{\mbox{{\rm stab}$\,$}}
	\newcommand{\lr}{\mbox{$\longrightarrow$}}
	
	\newcommand{\be}{\begin{equation}}
		\newcommand{\ee}{\end{equation}}
	
	\newtheorem{guess}{Theorem}[section]
	\newcommand{\bth}{\begin{guess}$\!\!\!${\bf }~}
		\newcommand{\eeth}{\end{guess}}
	\renewcommand{\bar}{\overline}
	\newtheorem{propo}[guess]{Proposition}
	\newcommand{\bpropo}{\begin{propo}$\!\!\!${\bf }~}
		\newcommand{\epropo}{\end{propo}}
	
	\newtheorem{lema}[guess]{Lemma}
	\newcommand{\blem}{\begin{lema}$\!\!\!${\bf }~}
		\newcommand{\elem}{\end{lema}}
	
	\newtheorem{defe}[guess]{Definition}
	\newcommand{\bdefe}{\begin{defe}$\!\!\!${\bf }~}
		\newcommand{\edefe}{\end{defe}}
	
	\newtheorem{coro}[guess]{Corollary}
	\newcommand{\bcor}{\begin{coro}$\!\!\!${\bf }~}
		\newcommand{\ecor}{\end{coro}}
	
	\newtheorem{rema}[guess]{Remark}
	\newcommand{\brem}{\begin{rema}$\!\!\!${\bf }~\rm}
		\newcommand{\erem}{\end{rema}}
	
	\newtheorem{exam}[guess]{Example}
	\newcommand{\beg}{\begin{exam}$\!\!\!${\bf }~\rm}
		\newcommand{\eeg}{\end{exam}}
	
	\newtheorem{notn}[guess]{Notation}
	\newcommand{\bnot}{\begin{notn}$\!\!\!${\bf }~\rm}
		\newcommand{\enot}{\end{notn}}

	\newcommand{\ch}{{\mathcal H}}
	\newcommand{\cf}{{\mathcal F}}
	\newcommand{\cd}{{\mathcal D}}
	\newcommand{\cR}{{\mathcal R}}
	\newcommand{\cv}{{\mathcal V}}
	\newcommand{\cn}{{\mathcal N}}
	\newcommand{\lra}{\longrightarrow}
	\newcommand{\ra}{\rightarrow}
	\newcommand{\blr}{\Big \longrightarrow}
	\newcommand{\da}{\Big \downarrow}
	\newcommand{\ua}{\Big \uparrow}
	\newcommand{\hra}{\mbox{{$\hookrightarrow$}}}
	\newcommand{\rt}{\mbox{\Large{$\rightarrowtail$}}}
	\newcommand{\dua}{\begin{array}[t]{c}
			\Big\uparrow \\ [-4mm]
			\scriptscriptstyle \wedge \end{array}}
	\newcommand{\ctext}[1]{\makebox(0,0){#1}}
	\setlength{\unitlength}{0.1mm}
	\newcommand{\cl}{{\mathcal L}}
	\newcommand{\cp}{{\mathcal P}}
	\newcommand{\ci}{{\mathcal I}}
	\newcommand{\bz}{\mathbb{Z}}
	\newcommand{\cs}{{\mathcal s}}
	\newcommand{\ce}{{\mathcal E}}
	\newcommand{\ck}{{\mathcal K}}
	\newcommand{\cz}{{\mathcal Z}}
	\newcommand{\cg}{{\mathcal G}}
	\newcommand{\ct}{{\mathcal T}}
	\newcommand{\cj}{{\mathcal J}}
	\newcommand{\cc}{{\mathcal C}}
	\newcommand{\ca}{{\mathcal A}}
	\newcommand{\cb}{{\mathcal B}}
	\newcommand{\cx}{{\mathcal X}}
	\newcommand{\co}{{\mathcal O}}
	\newcommand{\bq}{\mathbb{Q}}
	\newcommand{\bt}{\mathbb{T}}
	\newcommand{\bh}{\mathbb{H}}
	\newcommand{\br}{\mathbb{R}}
	\newcommand{\bl}{\mathbf{L}}
	\newcommand{\wt}{\widetilde}
	\newcommand{\im}{{\rm Im}\,}
	\newcommand{\bc}{\mathbb{C}}
	\newcommand{\bp}{\mathbb{P}}
	\newcommand{\ba}{\mathbb{A}}
	\newcommand{\spin}{{\rm Spin}\,}
	\newcommand{\ds}{\displaystyle}
	\newcommand{\tor}{{\rm Tor}\,}
	\newcommand{\bff}{{\bf F}}
	\newcommand{\bs}{\mathbb{S}}
	\def\ns{\mathop{\lr}}
	\def\nssup{\mathop{\lr\,sup}}
	\def\nsinf{\mathop{\lr\,inf}}
	\renewcommand{\phi}{\varphi}
	\newcommand{\tT}{{\widetilde{T}}}
	\newcommand{\tG}{{\widetilde{G}}}
	\newcommand{\tB}{{\widetilde{B}}}
	\newcommand{\tC}{{\widetilde{C}}}
	\newcommand{\tW}{{\widetilde{W}}}
	\newcommand{\tphi}{{\widetilde{\Phi}}}

	\title[Equivariant cohomology of odd--dimensional complex quadrics ]{Equivariant cohomology of odd--dimensional complex quadrics from a combinatorial point of view}

	\author[S. Kuroki]{Shintar{\^o} Kuroki}
	\address{Department of Applied Mathematics Faculty of Science, Okayama University of Science, 1-1 Ridai-Cho Kita-Ku Okayama-shi Okayama 700-0005, Okayama, Japan}
	\email{kuroki@ous.ac.jp}
	
	\author[B. Paul]{Bidhan Paul}
		\address{School of Mathematical Sciences, Tel Aviv University, Tel Aviv, 69978, Israel}
	\email{bidhanam95@gmail.com}
	

	\begin{abstract}
		This paper aims to determine the ring structure of the torus-equivariant cohomology  with integer coefficients of odd-dimensional complex quadrics by computing the graph equivariant cohomology of their corresponding GKM graphs. We show that this graph equivariant cohomology is generated by three types of subgraphs in the GKM graph, which satisfy five distinct types of relations. Furthermore, we investigate the relationship between the graph equivariant cohomology rings associated with odd- and even-dimensional complex quadrics.
	\end{abstract}

	\maketitle

\section{Introduction}
A {\it complex quadric} $Q_{N}$ is defined by the quadratic equation in the complex projective space $\mathbb{CP}^{N+1}$:
\begin{align*}
	Q_{N}:=\Big\{[z_1:\cdots:z_{N+2}]\in\mathbb{CP}^{N+1}~|~\sum_{i=1}^{N+2}z_i^2=0\Big\}.
\end{align*}
If $N$ is odd (resp.~even), then $Q_{N}$ is called an {\it odd-dimensional} (resp.~{\it even-dimensional}) {\it complex quadric}.
The standard $SO(N+2)$-action on  $\mathbb{CP}^{N+1}$ preserves this quadratic equation, leading to a maximal torus $T$-action on $Q_{N}$.

The cohomology ring of $Q_{N}$ has been studied. 
Over the integer coefficients, it is isomorphic to one of the following rings for $n\ge 1$ (see \cite[Excercise 68.3]{EKM08} or \cite{La74} if $N$ is even):
\begin{align*}
	H^{*}(Q_{N})\simeq 
	\left\{
	\begin{array}{lll}
		\mathbb{Z}[c,x]/\langle c^{n}-2x, x^{2} \rangle & {\rm if}\ N=2n-1, &  {\rm where}\ \deg c=2,\ \deg x=2n \\
		\mathbb{Z}[c,x]/\langle c^{2n}-2cx, x^{2} \rangle &  {\rm if}\ N=4n-2, &  {\rm where}\ \deg c=2,\ \deg x=4n-2 \\
		\mathbb{Z}[c,x]/\langle c^{2n+1}-2cx, x^{2}-c^{2n}x \rangle &  {\rm if}\ N=4n, &  {\rm where}\ \deg c=2,\ \deg x=4n
	\end{array}
	\right.
\end{align*}
In \cite{ku23}, the first author computed the $T^{n+1}$-equivariant cohomology ring of the even-dimensional complex quadric $Q_{2n}$ using GKM theory (also see \cite{b24}), where the $T^{n+1}:=(S^1)^{n+1}$-action on $Q_{2n}$, which is diffeomorphic to the zero locus of a quadratic equation in $\mathbb{CP}^{2n+1}$  i.e.
\begin{align*}
	Q_{2n}:=\Big\{[x_1:\cdots:x_{2n+2}]\in\mathbb{CP}^{2n+1}~|~\sum_{i=1}^{n+1}x_1x_{2n+3-i}=0\Big\},
\end{align*} is given by \begin{equation}\label{evencase_torusaction}
	[x_1:\cdots:x_{2n+2}]\cdot(t_1,\ldots,t_{n+1}):=[x_1t_1:\cdots:x_{n+1}t_{n+1}:t_{n+1}^{-1}x_{n+2}:\cdots:t_1^{-1}x_{2n+2}],
\end{equation}
where $(t_1,\ldots,t_{n+1})\in T^{n+1}$ and $[x_1:\cdots:x_{2n+2}]\in Q_{2n}$.  The action above is non-effective; the effective action is obtained by quotienting out the kernel (see \cite[Remark 2.2]{ku23}).
This provides a unified formula for the cohomology rings $H^{*}(Q_{4n-2})$ and $H^{*}(Q_{4n})$ using GKM graphs.
The approach also explains why two distinct relations appear in these rings through combinatorial interpretations.
Thus, GKM theory proves to be a powerful tool for studying (equivariant) cohomology rings of $T$-spaces. 

We briefly recall GKM theory here (see also the introduction of \cite{ku23}).
A {\it GKM manifold} is an equivariantly formal manifold $M^{2k}$ (i.e., $H^{odd}(M)=0$) endowed with a compact torus $T := (S^1)^n$ action. 
Goresky, Kottwitz, and MacPherson introduced this concept in \cite{gkm}, where 
the action ensures that the set of $0$- and $1$-dimensional orbits forms a graph structure. 
Motivated by their work, Guillemin and Zara introduced the concept of an \emph{abstract GKM graph} in \cite{gz01}, an abstract graph with labeled edges (see Section~\ref{odcqigg} for details). 
An abstract GKM graphs, which extend beyond their geometric motivations, 
have enriched the field known as \emph{GKM theory}.
In this paper, we study the $T^{n}$-equivariant cohomology of the odd-dimensional complex quadric $Q_{2n-1}$ from a GKM theoretical perspective. 

It is well-known that the $Q_{2n-1}$ is $T^{n}$-equivariantly diffeomorphic to the zero locus of a quadratic equation in $\mathbb{CP}^{2n}$ (see e.g.~\cite[Chapter V.1, 1.1 Theorem]{Se06}): 
\begin{align*}
	Q_{2n-1}:=\Big\{[z_1:\cdots:z_{2n+1}]\in\mathbb{CP}^{2n}~|~\sum_{i=1}^{n}z_iz_{2n+1-i}+z_{2n+1}^2=0\Big\}.
\end{align*}
In this paper, we consider $Q_{2n-1}$ as this manifold. 
From this definition, the natural $T^{n}$-action on $Q_{2n-1}$ is defined by:
\begin{equation}\label{introtorusaction}
	[z_1:\cdots:z_{2n+1}]\cdot(t_1,\ldots,t_{n}):=[z_1t_1:\cdots:z_{n}t_{n}:t_{n}^{-1}z_{n+1}:\cdots:t_1^{-1}z_{2n}:z_{2n+1}],
\end{equation}
where $(t_1,\ldots,t_{n})\in T^{n}$. 
Since $Q_{2n-1}$ is diffeomorphic to the oriented Grassmannian $SO(2n+1)/(SO(2n-1)\times SO(2))$ --the set of oriented $2$-planes through the origin in $\br^{2n+1}$-- this action is equivariantly diffeomorphic to the maximal torus $T^{n}$-action induced by restricting the transitive $SO(2n+1)$-action. 
Furthermore, as $SO(2n-1)\times SO(2)\subset SO(2n+1)$ is a maximal rank subgroup of $SO(2n+1)$, the set of $0$- and $1$-dimensional orbits of $T^{n}$-action forms a graph structure (see \cite{ghz06}). 
Therefore, the GKM graph of $Q_{2n-1}$ with $T^{n}$-action \eqref{introtorusaction} can be determined by labeling the edges with tangential representations.

By the formula for $H^{*}(Q_{2n-1})$ as above, we conclude that $H^{odd}(Q_{2n-1})=0$; hence, $Q_{2n-1}$ is an \emph{equivariantly formal} GKM manifold.
Consequently, the equivariant cohomology $H_{T^{n}}^{*}(Q_{2n-1})$ can be computed using the graph equivariant cohomology of its GKM graph, denoted by $\mathcal{GQ}_{2n-1}$. 
Let $\mathbb{Z}[\mathbf{M}, \mathbf{D}, Q]$ be the polynomial ring generated by some cohomology classes in $H^{*}(\mathcal{GQ}_{2n-1})$ and let $\mathfrak{I} \lhd \mathbb{Z}[\mathbf{M}, \mathbf{D}, Q]$ be the ideal generated by {five types} of relations (cf. Section~\ref{secmainth}) among the elements in $\mathbf{M}$, $\mathbf{D}$, and $Q$. We  define 
\[
\mathbb{Z}[\mathcal{GQ}_{2n-1}] := \mathbb{Z}[\mathbf{M}, \mathbf{D}, Q] / \mathfrak{I}.
\]
The main theorem of this paper, precisely presented in Section~\ref{secmainth}, is as follows: 
\begin{guess}[Theorem~\ref{mainth}]
	\label{main}
	There exist the following isomorphisms as a ring:
	\begin{align*}
		H^{*}_{T^{n}}(Q_{2n-1})\simeq H^{*}(\mathcal{GQ}_{2n-1})\simeq \mathbb{Z}[\mathcal{GQ}_{2n-1}],
	\end{align*} 
	where the first isomorphism comes from the fact that $Q_{2n-1}$ is an equivariantly formal GKM manifold.
\end{guess}
On the other hand, there is a non-effective $T^{n}$-action on $Q_{2n-2}$, which is restricted from that on $Q_{2n-1}$.
In the second part of this paper (Section~\ref{sect:7}), we study the homomorphism $H^{*}_{T}(Q_{2n-1})\to H_{T}^{*}(Q_{2n-2})$ using GKM theory. 
We also compute the equivariant cohomology of the non-effective $T^{n}$-action on $Q_{2n-2}$, obtaining the following result by combining Lemma~\ref{key-lemma} and Theorem~\ref{thm-non-effective-action} (see Section~\ref{sect:7} for details):
\begin{guess}
	\label{main2}
	There exist the following isomorphisms as a ring:
	\begin{align*}
		H^{*}_{T^{n}}(Q_{2n-2})\simeq H^{*}(\mathcal{GQ}_{2n-2})\simeq \mathbb{Z}[\mathcal{M}, \mathcal{D}, X]/\mathcal{I},
	\end{align*} 
	where $X$ and the elements in $\mathcal{M}$ are of degree 2, and $\Delta_K'\in\mathcal{D}$ for $K\subset \cv_{2n-2}$ is of degree $4n-2|K|-2$.
\end{guess}
Note that the generators $\mathcal{M}$ and $\mathcal{D}$ in Theorem~\ref{main2} also appear in the equivariant cohomology of $Q_{2n-2}$ with the effective $T^{n}$-action as the equivariant Thom classes, as shown in \cite{ku23}. 
Theorem~\ref{main2} reveals that in the non-effective case, we also need an additional generator $X$, which is not induced from the equivariant Thom classes.
Motivated by this phenomenon, in the final part of this paper (Section~\ref{sect:8}), we further study the non-effective $T^{1}$-actions on $\mathbb{CP}^1$, and obtain  the following result:
\begin{guess}
	\label{main3}
	Let $\varphi_{n}$ be the $n$-times rotated $T^{1}$-action on $\mathbb{CP}^1$ for $n\ge 0$, and let $H_{\varphi_{n}}^{*}(\mathbb{CP}^{1})$ be its equivariant cohomology.
	Then,  there is the following ring isomorphism:
	\begin{align*}
		H_{\varphi_{n}}^{*}(\mathbb{CP}^{1})\simeq \mathbb{Z}[\tau_{p},\tau_{q}, \alpha]/\langle \tau_{p}\tau_{q}, n\alpha-\tau_{p}+\tau_{q} \rangle,
	\end{align*} where 
	$\tau_{p}$ (resp.~$\tau_{q}$) is the equivariant Thom class associated with the fixed points $p=[1:0]$ (resp.~$q=[0:1]$), and
	$\alpha$ is the pull-back of the generator of $H^{*}(BT^{1})\simeq \mathbb{Z}[\alpha]$.
\end{guess}
Comparing the effective ($n=1$) and the non-effective ($n\not=1$) cases, we also observe that there is an additional generator for the non-effective $T^{1}$-actions on $\mathbb{CP}^1$, which may not be induced from the equivariant Thom classes.

{\bf This article is structured as follows:}
In Section \ref{odcqigg}, we compute the GKM graph $\mathcal{GQ}_{2n-1}$ of the effective $T^n$-action on $Q_{2n-1}$. In Section 3, we introduce the graph equivariant cohomology  $H^*(\mathcal{GQ}_{2n-1})$ and define the generators $M_v,$ $\Delta_K$ and $Q$, studying their properties. 
In Section 4, we present the five relations among $M_v,~\Delta_K$ and $Q$. 
The main theorem (Theorem \ref{mainth}) is proved in Section \ref{secmainth}. In Section \ref{sect:6}, the ordinary cohomology ring of $Q_{2n-1}$ is studied from a GKM theoretical perspective. 
The comparison of two graph equivariant cohomologies induced from even- and odd- dimensional complex quadrics is studied in Section \ref{sect:7}.
Finally, in Section \ref{sect:8}, we provide a GKM description for non-effective $T^1$-actions on $Q_{1}\simeq \mathbb{CP}^1$.

\section{GKM graphs of odd-dimensional complex quadrics $Q_{2n-1}$}
\label{odcqigg}
We use the symbol $(Q_{2n-1}, T^{n})$ to denote the $T^n$-action on $Q_{2n-1}$ defined by \eqref{introtorusaction}. 
We first describe the GKM graph of $(Q_{2n-1}, T^{n})$. 
For the basics of GKM manifolds and GKM graphs, see the paper \cite{gz01}. 
In this article, we identify the cohomology ring $H^*(BT^n)$ as the following polynomial ring generated by degree 2 generators $\alpha_1,\ldots,\alpha_n$:
\begin{equation}\label{2.2}
	H^*(BT^n)\simeq\bz[\alpha_1,\ldots,\alpha_n].
\end{equation}
Here, one may consider the generator $\alpha_j$, for $j=1,\ldots,n$, as the $j$-th coordinate projection $pr_j:T^n\to S^1$. Namely, we often use the following identifications:
\begin{equation}
	\label{identificaitons}
	H^2(BT^n)\simeq \mbox{Hom}(T^n,S^1)\simeq (\mathfrak{t}^n_{\bz})^*\simeq \bz^n,
\end{equation}
where $\mathfrak{t}^n_{\bz}$ is the lattice of the Lie algebra of $T^n$ and $(\mathfrak{t}^n_{\bz})^{*}$ is its dual.

\subsection{The GKM graph of the $T^n$-action on $Q_{2n-1}$}
\label{sect:2.1}
We now compute the GKM graph of $(Q_{2n-1}, T^{n})$. It is easy to check that the $T^n$-fixed points of $Q_{2n-1}$ are 
\[Q_{2n-1}^T=\{[e_i]\,:\,1\leq i\leq 2n\},\] where $[e_i]=[0:\cdots:0:1:0:\cdots:0]\in \mathbb{CP}^{2n}$ is the $i$-th coordinate with $1$ and $0$ otherwise.
Furthermore, the $T^{n}$-invariant $2$-spheres of $Q_{2n-1} $ are of the following two types:
\[\begin{split}
	&[z_i:z_j]:= [0:\cdots:0:z_i:0:\cdots:0:z_j:0:\cdots:0]\in Q_{2n-1},~~~ \text{if $i+j\neq 2n+1$};\\
	&[z_i:z_j]:= [0:\cdots:0:z_i:0:\cdots:0:z_j:0:\cdots:z_{2n+1}]\in Q_{2n-1},~~~ \text{if $i+j= 2n+1$}.
\end{split}\]
Note that the second type as above satisfies the equation $z_{i}z_{j}+z_{2n+1}^{2}=0$.
Therefore, every pair of $[e_i], [e_j]\in Q_{2n-1}^T$ becomes the fixed points of a $T^{n}$-invariant $2$-sphere. 
Hence, we associate the following graph $\Gamma_{2n-1} := (\mathcal{V}_{2n-1},~ \mathcal{E}_{2n-1})$ with the pair $(Q_{2n-1}, T^{n})$:
\begin{description}
	\item[Vertices] The set of vertices $\mathcal{V}_{2n-1}$ consists of $Q_{2n-1}^T$, where we may identify the vertices $Q_{2n-1}^T=\{[e_{i}]\ |\ 1\le i\le 2n\}$ with the set of numbers $[2n]:=\{1,\ldots,2n\}$;
	\item[Edges] The set of edges $\mathcal{E}_{2n-1}$ consists of $ij$ for every $i,j\in \mathcal{V}_{2n-1}$.
\end{description}
Combinatorially, $\Gamma_{2n-1}$ is just the complete graph with $2n$ vertices.
\begin{rema}
	Here we consider the edges to be undirected. This means that for each undirected edge in the graph, there are two directed edges, one in each direction.
\end{rema} 
\begin{figure}[H]
	\centering
	\begin{tikzpicture}[row sep=small, column sep=small]
		\node[draw, circle, fill=black, inner sep=1.5pt, label=right:5] (A) at (1, 0.5) {};
		\node[draw, circle, fill=black, inner sep=1.5pt, label=right:4] (B) at (1,-0.5) {};
		\node[draw, circle, fill=black, inner sep=1.5pt, label=left:3] (C) at (-1, 0.5) {};
		\node[draw, circle, fill=black, inner sep=1.5pt, label=left:2] (D) at (-1, -0.5) {};
		\node[draw, circle, fill=black, inner sep=1.5pt, label=6] (E) at (0, 1) {};	
		\node[draw, circle, fill=black, inner sep=1.5pt, label=below:1] (F) at (0, -1) {};		
		\draw (A) -- node[midway, below, sloped]  {} (B);
		\draw (B) -- node[midway, below right, sloped] {} (D);
		\draw (C) -- node[midway, below left, sloped] {} (A);
		\draw (A) -- node[midway, below, sloped]  {} (E);
		\draw (A) -- node[midway, below, sloped]  {} (F);
		\draw (B) -- node[midway, below, sloped]  {} (F);
		\draw (B) -- node[midway, below, sloped]  {} (E);
		\draw (F) -- node[midway, below, sloped]  {} (D);
		\draw (F) -- node[midway, below, sloped]  {} (C);
		\draw (D) -- node[midway, below, sloped]  {} (C);
		\draw (C) -- node[midway, below, sloped]  {} (E);
		\draw (D) -- node[midway, below, sloped]  {} (E);
		\draw (F) -- node[midway, below, sloped]  {} (E);
		\draw (B) -- node[midway, below, sloped]  {} (C);
		\draw (A) -- node[midway, below, sloped]  {} (D);
	\end{tikzpicture}
	
	\caption{The above graph shows $\Gamma_5$ for $n=3$ induced from $(Q_{5},T^3)$.} \label{fig:gamma4diag}
\end{figure}
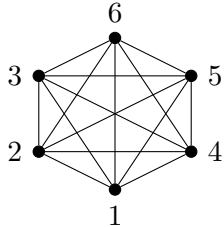

\begin{rema}
	\label{def_comb_GKMgraph}
	For simplicity, we often write $j$ as $\overline{i}$ when $i+j=2n+1$. Hence the set of vertices and edges in $\Gamma_{2n-1}$ can be rewritten as follows: 
	\begin{align}
		\mathcal{V}_{2n-1}&:=\{1,\ldots,n,\overline{n},\overline{n-1},\ldots, \overline{1}\}; \\
		\mathcal{E}_{2n-1}&:=\{ij\,:\,i,j\in\mathcal{V}_{2n-1}\text{ such that } j\neq i\}.
	\end{align}	
\end{rema}

We next define an \emph{axial function} $\alpha :\ce_{2n-1}\to H^2(BT^{n})$ on the edges by computing the tangential representations around the fixed points. 
Since the computation for the other vertices follows a similar method, we will compute the tangential representation only around the fixed point $[e_1]=[1:0:\cdots:0]$.
Around $[e_1]\in Q_{2n-1}^{T}$, the tangential representation is given by
\[[1:z_2:\cdots:z_{2n}:z_{2n+1}]\mapsto [1:t_1^{-1}t_2z_2:\cdots:t_1^{-1}t_nz_n:t_1^{-1}t_n^{-1}z_{n+1}:\cdots:t_1^{-2}z_{2n}:t_1^{-1}z_{2n+1}]. \]
We notice that the quadric equation for the $z_{2n}$ coordinate satisfies the following relation:
\[\Big(z_{2n}=-\sum_{i=2}^{n}z_iz_{2n+1-i}-z_{2n+1}^2\Big)\mapsto\Big( t_1^{-2}z_{2n}=-\sum_{i=2}^{n}(t_1^{-1}t_i)z_i(t_1^{-1}t_i^{-1})z_{2n+1-i}-(t_1^{-1}z_{2n+1})^2\Big).\]
This shows that the representation on the $z_{2n}$ coordinate is automatically determined by the representations on the other corrdinates.
Therefore, the tangential representation around the fixed point $[e_{1}]$ can be obtained from the representations on the coordinates except $z_{1}$ and $z_{2n}$. More precisely, it splits into the following complex $1$-dimensional irreducible representations:
\[T_{[e_{1}]}Q_{2n-1}\simeq \bigoplus_{i=2}^{2n-1} V(-\alpha_{1}+\alpha_{i})\oplus V(-\alpha_{1}),\]
where $V(\beta)$ is the complex $1$-dimensional representation defined by the homomorphism $\beta:T^{n}\to S^{1}$.
Recall that $\alpha_{i}\in \mbox{Hom}(T^{n},S^{1})$ for $i\in [n]$ is the representation corresponding to the $i$-th coordinate projection $pr_i :T^n\to S^1$ (see \eqref{2.2} and \eqref{identificaitons}).  
For $i\in \{n+1,\ldots,2n\}$, we regard 
\begin{equation}
	\label{label_opposite_vertex}
	\alpha_i:=-\alpha_{\overline{i}}~~\text{ for $i\in \{n+1,\ldots,2n\}$ and $\overline{i}=2n+1-i$.}
\end{equation}
Similarly, each tangential representation around a fixed point decomposes into complex $1$-dimensional irreducible representations. Furthermore, each $1$-dimensional irreducible representation corresponds to the tangential representation on the fixed point of the invariant $2$-sphere.
Therefore, we can define the following axial function on the edges: 
\begin{equation}
	\label{axial_function_of_odd-cpx-quadrics}
	\alpha:\ce_{2n-1}\to H^2(BT^n)
\end{equation} 
which takes 
\begin{align*}
	&\alpha(ij)=-\alpha_{i}+\alpha_{j} \text{ for $j\not=i,\overline{i}$},\\ 
	&\alpha(i\overline{i})=-\alpha_{i}.
\end{align*}  
More precisely, using the notation \eqref{label_opposite_vertex}, 
we have the following assignments:
\begin{itemize}
	\item $\alpha(ij)=-\alpha_i+\alpha_j$ for $1\leq i\neq j\leq n$;
	\item $\alpha(i\overline{j})=-\alpha_{i}+\alpha_{\overline{j}}=-\alpha_i-\alpha_j$ for $1\leq i\neq j\leq n$;
	\item $\alpha(\overline{i}j)=-\alpha_{\overline{i}}+\alpha_{j}=+\alpha_i+\alpha_j$ for $1\leq i\neq j\leq n$;
	\item $\alpha(\overline{i}\overline{j})=-\alpha_{\overline{i}}+\alpha_{\overline{j}}=\alpha_i-\alpha_j$ for $1\leq i\neq j\leq n$;
	\item $\alpha(i\overline{i})=-\alpha_i$ for $1\leq i\leq n$, 
	\item $\alpha(\overline{i}i)=-\alpha_{\overline{i}}=\alpha_i$ for $1\leq i\leq n$.
\end{itemize}
Note that the relation $\alpha(ij)=-\alpha(ji)$ holds for all $ij\in \mathcal{E}_{2n-1}$.

\begin{notn}
	Henceforth, the symbol $\mathcal{GQ}_{2n-1}$ represents
	the GKM graph $(\Gamma_{2n-1},\alpha)$, where $\Gamma_{2n-1}=(\cv_{2n-1},\ce_{2n-1})$ in Remark~\ref{def_comb_GKMgraph}, and the axial function $\alpha:\mathcal{E}_{2n-1}\to H^{2}(BT^{n})$ is defined by \eqref{axial_function_of_odd-cpx-quadrics}.
\end{notn}

\begin{rema}
	\label{comapre_symp_geom}
	Note that the $T^{n}$-action on $Q_{2n-1}\subset \mathbb{CP}^{2n}$ 
	is a Hamiltonian torus action obtained by restricting the $T^{n}$-action on $\mathbb{CP}^{2n}$. 
	It is easy to check that its moment-map image is the $n$-dimensional crossed polytope, i.e., ${\rm Conv}\{\pm e_{i}\ |\ i=1,\ldots, n\}\subset (\mathfrak{t}^{n})^{*}.$ However, the GKM graph $\Gamma_{2n-1}$ of $Q_{2n-1}$ is the complete graph with $2n$ vertices, not the one-skeleton of the crossed polytope.
	On the other hand, the restricted (non-effective) $T^{n}$-action on $Q_{2n-2}(\subset Q_{2n-1})$ is also a Hamiltonian torus action whose moment-map image is the $n$-dimensional crossed polytope.
	By Section~\ref{sect:7.1} (also see \cite{ku23}), we see that the GKM graph of $Q_{2n-2}$ is indeed the one-skeleton of the crossed polytope.
\end{rema}

\subsection{Examples of $\mathcal{GQ}_{2n-1}$ for $n=2,3$.}
\label{sect:2.2}
In this subsection, we present two low-dimensional examples of $\mathcal{GQ}_{2n-1}$. 
Figure~\ref{lowest-example} depicts the labeled graph on the edges of $\Gamma_{3}$ using axial functions, i.e., the GKM graph $\mathcal{GQ}_{3}$.

\begin{figure}[H]
	\begin{tikzpicture}
		\begin{scope}[xscale=0.5, yscale=0.5]
			\coordinate (1) at (-3,-3);
			\coordinate (2) at (-3,3);
			\coordinate (3) at (3,3);
			\coordinate (4) at (3,-3);
			
			\fill(1) circle (5pt);
			\node[left] at (1) {$1$};
			\fill(2) circle (5pt);
			\node[left] at (2) {$2$};
			\fill(3) circle (5pt);
			\node[right] at (3) {$4=\overline{1}$}; 
			\fill(4) circle (5pt);
			\node[right] at (4) {$3=\overline{2}$}; 
			
			\draw[thick, decoration={markings,
				mark=at position 0.5 with \arrow{>}},
			postaction=decorate](1)--(2);
			\node[left] at (-3,0) {$-\alpha_{1}+\alpha_{2}$};
			
			\draw[thick, decoration={markings,
				mark=at position 0.2 with \arrow{>}},
			postaction=decorate](1)--(3);
			\node[right] at (-2,-2) {$-\alpha_{1}$};
			
			\draw[thick, decoration={markings,
				mark=at position 0.5 with \arrow{>}},
			postaction=decorate](1)--(4);
			\node[right] at (3,0) {$\alpha_{1}-\alpha_{2}$};
			
			\draw[thick, decoration={markings,
				mark=at position 0.5 with \arrow{>}},
			postaction=decorate](3)--(2);
			\node[above] at (0,3) {$\alpha_{1}+\alpha_{2}$};
			
			\draw[thick, decoration={markings,
				mark=at position 0.2 with \arrow{>}},
			postaction=decorate](2)--(4);
			\node[right] at (-2,2) {$-\alpha_{2}$};
			
			\draw[thick, decoration={markings,
				mark=at position 0.5 with \arrow{>}},
			postaction=decorate](3)--(4);
			\node[below] at (0,-3) {$-\alpha_{1}-\alpha_{2}$};
			
		\end{scope}
	\end{tikzpicture}
	\caption{The above figure shows the GKM graph $\mathcal{GQ}_{3}$. This satisfies $\alpha(ij)=-\alpha(ji)$; therefore, we omit the labels in the opposite directions.}
	\label{lowest-example}
\end{figure}
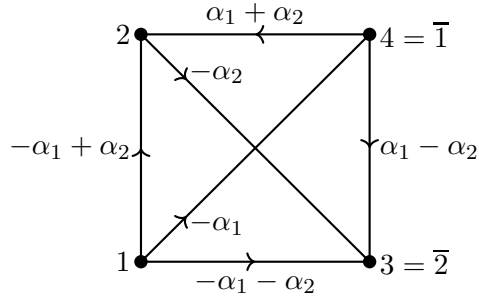

On the other hand, the following table represents the axial functions on the edges of $\Gamma_{5}$.
For $i,j\in \{1,\ldots, 6\}$, 
the $ij$ box corresponds to the edge $ij\in \mathcal{E}_{5}$ of $\Gamma_{5}$ in Figure~\ref{fig:gamma4diag} and its value indicates $\alpha(ij)$.
From this table, we can immediately reconstruct the labeled graph, i.e., the GKM graph $\mathcal{GQ}_{5}$, from this table. 
Therefore, this table also represents the GKM graph.

\begin{table}[H]
	\begin{tabular}{|c|c|c|c|c|c|c|}
		\hline
		&1&2&3&$4=\overline{3}$&$5=\overline{2}$&$6=\overline{1}$\\
		\hline
		$1$&N/A &$-\alpha_1+\alpha_2$ &$-\alpha_1+\alpha_3$ &$-\alpha_1-\alpha_3$ &$-\alpha_1-\alpha_2$ &$-\alpha_1$\\
		\hline
		$2$&$\alpha_1-\alpha_2$&N/A&$-\alpha_2+\alpha_3$&$-\alpha_2-\alpha_3$&$-\alpha_2$&{$-\alpha_1-\alpha_2$}\\
		\hline
		$3$&$\alpha_1-\alpha_3$ &$\alpha_2-\alpha_3$ &N/A &$-\alpha_3$ &{$-\alpha_2-\alpha_3$} &{$-\alpha_1-\alpha_3$}\\
		\hline
		$\overline{3}$ &$\alpha_1+\alpha_3$ &$\alpha_2+\alpha_3$ &$\alpha_3$ &N/A&$-\alpha_2+\alpha_3$ &$-\alpha_1+\alpha_3$\\
		\hline
		$\overline{2}$ &$\alpha_1+\alpha_2$ &$\alpha_2$ &{$\alpha_2+\alpha_3$} &$\alpha_2-\alpha_3$ &N/A&$-\alpha_1+\alpha_2$\\	
		\hline
		$\overline{1}$	&$\alpha_1$&$\alpha_1+\alpha_2$ &$\alpha_1+\alpha_3$ &$\alpha_1-\alpha_3$&$\alpha_1-\alpha_2$ &$N/A$\\
		\hline
		
	\end{tabular}
	\caption{This table shows the axial functions on the edges of the graph $\Gamma_{5}$ in Figure~\ref{fig:gamma4diag}, where N/A indicates that there are no edges connecting $ii$ for $i\in \mathcal{V}_{5}$.
		This information also provides the same details of a GKM graph, as illustrated in Figure~\ref{lowest-example}.}
	\label{table1}
\end{table}

\subsection{GKM subgraphs} 
\label{sect:2.3}
In this subsection, we recall the notion of GKM subgraphs of a GKM graph $(\Gamma, \alpha)$. Let $\Gamma' \subset \Gamma$ be a subgraph. We use the following notation:
\begin{itemize}
	\item $\cv(\Gamma)$ and $\ce(\Gamma)$ (resp.\ $\cv(\Gamma')$ and $\ce(\Gamma')$) denote the sets of vertices and edges of $\Gamma$ (resp.\ of $\Gamma'$), respectively.
	\item For a vertex $v\in\cv(\Gamma')\subset \cv(\Gamma)$, let $\ce_v(\Gamma)$ (resp. $\ce_v(\Gamma')$) is the set of outgoing edges in $\Gamma$ (resp.~ in $\Gamma'$).
\end{itemize}
Let $e=vw\in \ce(\Gamma)$ be an edge. We define a {\it connection along} $e$  to be a bijection
\begin{equation}\label{conn}
	\nabla_e : \ce_v(\Gamma)\to \ce_w(\Gamma)
\end{equation}  
such that $\nabla_e(e)=\overline{e}$, where $\overline{e}$ is the edge with direction opposite to $e$. Furthermore, we define a \emph{connection} on graph $\Gamma$  to be a collection $\nabla:=\{\nabla_e\}_{e\in\ce(\Gamma)}$ satisfying $\nabla_e^{-1}=\nabla_{\overline{e}}$ for every $e\in\ce(\Gamma)$.
%
\begin{defe}[GKM subgraph]\label{gkmsubgraph}
	Let $\Gamma'$ be a subgraph of $\Gamma$ and $\alpha'$ be the restricted axial function on its edges, i.e., $\alpha':=\alpha|_{\mathcal{E}(\Gamma')}$. 
	Let $\nabla$ be the connection on $\Gamma$.  
	We call $(\Gamma',\alpha')$ a {\it GKM subgraph} if it is closed under $\nabla$. 
	More precisely, for all $e\in\ce(\Gamma')$ with $i(e) = v$ and $t(e) = w\in \cv(\Gamma')$, the restricted bijection \[(\nabla_e)|_{\ce_v(\Gamma')} :\ce_v(\Gamma')\to \ce_w(\Gamma')\] is well-defined, 
	In this case, $(\Gamma',\alpha')$ where $\alpha':=\alpha|_{\Gamma'}$ is again a GKM graph.
\end{defe}

\begin{exam}
	Consider the GKM graph of \(\mathbb{CP}^3\) with the torus action of $T:=(S^1)^4$ given by 
	\[[x_1:x_2:x_3:x_4]\cdot(t_1,t_2,t_3,t_4):=[t_1x_1:t_2x_2:t_3x_3:t_4x_4],\] for $[x_1:x_2:x_3:x_4]\in \mathbb{CP}^3$ and $(t_1,t_2,t_3,t_4)\in T$. The torus-fixed points  are  denoted by \(p_1, ~p_2,~ p_3,~ p_4\) and edges  are labeled by the corresponding weights \(\alpha_i - \alpha_j\). 
	A subgraph corresponding to \(\mathbb{CP}^2\) inside \(\mathbb{CP}^3\), consisting of the vertices \(p_1,~ p_2, ~p_3\), is highlighted in thick line in Figure~\ref{fig:CP2-in-CP3}.
	\begin{figure}[H]
		\centering
		\begin{tikzpicture}[scale=2, every node/.style={font=\small}]
			\node[circle,fill=black,inner sep=1.5pt,label=above left:{$p_1$}] (p1) at (0,0) {};
			\node[circle,fill=black,inner sep=1.5pt,label=above right:{$p_2$}] (p2) at (1.2,0) {};
			\node[circle,fill=black,inner sep=1.5pt,label=below:{$p_3$}] (p3) at (0.6,-0.9) {};
			\node[circle,fill=black,inner sep=1.5pt,label=below right:{$p_4$}] (p4) at (1.8,-0.9) {};
			\draw[dashed ] (p1) -- node[sloped, above]{$\alpha_1 - \alpha_2$} (p2);
			\draw[dashed] (p1) -- node[sloped, below]{$\alpha_1 - \alpha_3$} (p3);
			\draw[dashed] (p1) -- node[pos=0.7,sloped, above]{$\alpha_1 - \alpha_4$} (p4);
			\draw[dashed] (p2) -- node[pos=0.54,sloped, above]{$\alpha_2 - \alpha_3$} (p3);
			\draw[dashed] (p2) -- node[sloped, above]{$\alpha_2 - \alpha_4$} (p4);
			\draw[dashed] (p3) -- node[below]{$\alpha_3 - \alpha_4$} (p4);
			\draw[ thick] (p1) -- (p2);
			\draw[ thick] (p2) -- (p3);
			\draw[ thick] (p3) -- (p1);
		\end{tikzpicture}
		\caption{GKM subgraph of $\mathbb{CP}^2$ (thick) as a subset of the GKM graph of $\mathbb{CP}^3$.}
		\label{fig:CP2-in-CP3}
	\end{figure}
\end{exam}

\section{Generators of $H^*(\mathcal{GQ}_{2n-1})$}
\label{gen}
The \emph{graph equivariant cohomology} of the GKM graph $\mathcal{GQ}_{2n-1}$ is defined as follows:
\begin{equation}\label{congrel}
	H^*(\mathcal{GQ}_{2n-1})=\{f:\cv_{2n-1}\to H^{*}(BT^n) ~:~f(i)-f(j)\equiv 0~\mbox{mod}~ \alpha(ij) ~\mbox{for}~ij\in\ce_{2n-1}\}.
\end{equation}
The equation $f(i)-f(j)\equiv 0~\mbox{mod}~\alpha(ij)$ in \eqref{congrel} is often referred to as a \emph{congruence relation}. 
Note that $	H^*(\mathcal{GQ}_{2n-1})$ has a graded $H^*(BT^n)$-algebra structure induced by the graded algebra structure of $\bigoplus_{k\geq 0}\big(\bigoplus_{i\in\cv_{2n-1}} H^k(BT^n)\big)$. 
This algebraic structure is also induced by the injective homomorphism 
\begin{equation}\label{eqn3.10}
	\vartheta: H^*(BT^n)\to H^*(\mathcal{GQ}_{2n-1})
\end{equation}
such that the image of $x\in H^*(BT^n)$ (i.e., $\vartheta(x):\cv_{2n-1}\to H^*(BT^n)$) is defined by the function \[\vartheta(x)(i)=x \text{    ~~~~~for all $i\in \cv_{2n-1}$}.\]
Furthermore, $H^*(\mathcal{GQ}_{2n-1})$ also acquires an $H^*(BT^n)$-module structure  induced by the above injective homomorphism $\vartheta$. 

\begin{lema}\label{Lemma3.1}
	For the $T^n$-action on $Q_{2n-1}$,  we have the following graded $H^*(BT^n)$-algebra isomorphism:
	\begin{equation}
		H^*_{T^n}(Q_{2n-1})\simeq H^*(\mathcal{GQ}_{2n-1}).
	\end{equation}
\end{lema}
\begin{proof}
	It is straightforward to verify 
	that for $x=[z_1:\cdots:z_{2n+1}]\in Q_{2n-1}$, the isotropy subgroup is,  for some $k\leq n$, given by:
	\[(T^n)_x\simeq \begin{cases}
		T^k &\text{if $z_{2n+1}\neq 0$},\\
		T^k\times \bz_2 &\text{if $z_{2n+1}=0$}.
	\end{cases}\]
	Furthermore, the odd-cohomologies of $Q_{2n-1}$ vanish (see Section~\ref{sect:6}). Therefore, the isomorphism holds by \cite{fp11},  (See, Theorem \ref{FPT} and also the proof of Lemma \ref{FPT2}).
	

\end{proof}

Therefore, to compute the $T^n$-equivariant cohomology ring $H^*_{T^n}(Q_{2n-1})$ of $Q_{2n-1}$, 
it sufficies to compute the graph equivariant cohomology $H^*(\mathcal{GQ}_{2n-1})$. 
In this section, we introduce three types of elements of $H^*(\mathcal{GQ}_{2n-1})$, denoted $M_{v}$, $Q$ and $\Delta_{K}$ which serves as its generators (see Section \ref{secmainth}), and describe the ring structure of $H^*(\mathcal{GQ}_{2n-1})$ in terms of these generators and relations based on the combinatorial data of $\mathcal{GQ}_{2n-1}$.

\subsection{Degree $2$ generators}
\label{sect:3.1}
In this subsection, we will define two types of degree two {elements}, called $M_{v}$ for $v\in \mathcal{V}_{2n-1}$ and $Q$.   Later, in Lemma~\ref{surjmainth}, we prove that these elements together with those defined in Section~\ref{sec:high_deg_gen} will  generate $H^{*}(\mathcal{GQ}_{2n-1})$.

\begin{defe}\label{M_v}
	Let $v\in \cv_{2n-1}$ be a vertex. 
	We define the function $M_v:\cv_{2n-1}\to H^2(BT^n)$ by 
	\begin{equation}\label{mapmv}
		M_v(i)=\begin{cases}
			0 & \text{ if}~~i=v;\\
			2\alpha_{v}& \text{ if}~~i=\overline{v};\\
			\alpha(iv)=-\alpha_i+\alpha_v& \text{ if}~~i\not=v, \overline{v};
		\end{cases}
	\end{equation}
	where $\alpha_{i}=-\alpha_{\overline{i}}$ for $i\in \{n+1,\ldots, 2n\}$ (see \eqref{label_opposite_vertex}).
\end{defe}

\begin{propo}
	For every $v\in\cv_{2n-1}$, $M_v\in H^2(\mathcal{GQ}_{2n-1})$.
\end{propo}

\begin{proof}
	To show this, it is enough	to check that $M_v$ satisfies the congruence relations \eqref{congrel} for every edge $ij\in\ce_{2n-1}$. 
	We verify the above condition on a case-by-case basis:
	\begin{enumerate}[(a)]
		\item \underline{If $i=v$ and $j=\overline{v}$}: 
		Since $\alpha(v\overline{v})=-\alpha_v$, we have	\[M_v(i)-M_v(j)=-2\alpha_v\equiv 0~~\mbox{mod}~~\alpha(ij).\] 
		\item \underline{If $i=v$ and $j\not=v,\overline{v}$}:  
		We have	$M_v(i)-M_v(j)=-\alpha(jv)
		\equiv 0~~\mbox{mod}~~\alpha(ij)$. 
		\item \underline{If $i=\overline{v}$ and $j\not=v,\overline{v}$}:
		We have that \[M_v(i)-M_v(j)=2\alpha_v-\alpha(jv)
		=\alpha_v+\alpha_j\equiv 0~~\mbox{mod}~~\alpha(ij).\] 
		\item \underline{If $i\not=v,\overline{v}$ and $j\not=v,\overline{v}$}:    
		We have \[M_v(i)-M_v(j)=(-\alpha_i+\alpha_v)-(-\alpha_j+\alpha_v)=-\alpha_i+\alpha_j\equiv 0~~\mbox{mod}~~\alpha(ij).\] 
	\end{enumerate}
	Therefore, $M_v\in H^2(\mathcal{GQ}_{2n-1})$ for all $v\in\cv_{2n-1}$.
\end{proof}

		
		\begin{exam} 
			For $n=2$ (Figure~\ref{lowest-example}), the following figure shows the element $M_4\in H^{2}(\mathcal{GQ}_{3})$. 
			This may be regarded as the full subgraph spanned by the vertices $1,2,3\in \mathcal{V}_{3}$ in $\Gamma_{3}$. 
			Note that $M_{v}(j)$ for $j\not=v,\overline{v}$ coincides with the normal axial function of this full subgraph. 
			Moreover, we can easily check that $M_{v}(\overline{v})$ is the unique element that  satisfies the congruence relations with the other $M_{v}(j)$'s (cf. \cite[Proposition 3.5]{ku23}).
			
			\begin{figure}[H]
				\begin{tikzpicture}
					\begin{scope}[xscale=0.5, yscale=0.5]
						\coordinate (1) at (-3,-3);
						\coordinate (2) at (-3,3);
						\coordinate (3) at (3,3);
						\coordinate (4) at (3,-3);
						
						\fill(1) circle (5pt);
						\node[left] at (1) {$M_{4}(1)=-2\alpha_{1}$};
						\fill(2) circle (5pt);
						\node[left] at (2) {$M_{4}(2)=-\alpha_{2}+\alpha_{4}=-\alpha_{2}-\alpha_{1}$};
						\fill(3) circle (5pt);
						\node[right] at (3) {$M_{4}(4)=0$}; 
						\fill(4) circle (5pt);
						\node[right] at (4) {$M_{4}(3)=-\alpha_{3}+\alpha_{4}=\alpha_{2}-\alpha_{1}$}; 
						
						\draw[thick](1)--(2);
						
						\draw[dashed](1)--(3);
						
						\draw[thick](1)--(4);
						
						\draw[dashed](3)--(2);
						
						\draw[thick](2)--(4);
						
						\draw[dashed](3)--(4);
						
					\end{scope}
				\end{tikzpicture}
				\caption{The element $M_{4}\in H^{2}(\mathcal{GQ}_{3})$.}
				\label{M_4_of_lowest-example}
			\end{figure}
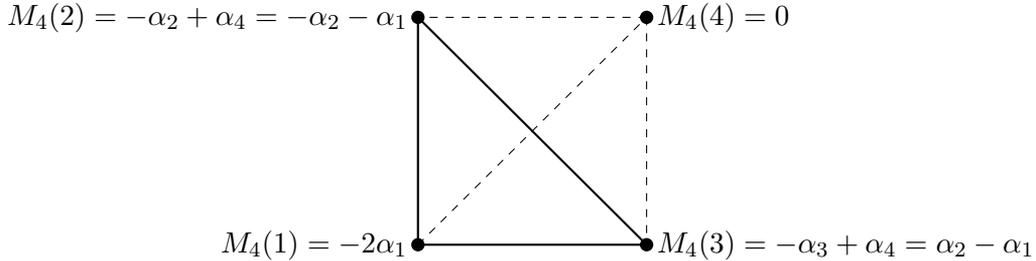
		\end{exam}

		\begin{exam} 
			By using the table, we can represent all $M_v$'s at once.
			For example, for $n=3$, the following table 
			represents $M_v$ for all $v\in\cv_5$.
			\begin{center}
				\begin{tabular}{|c|c|c|c|c|c|c|}
					\hline
					&1&2&3&$4=\overline{3}$&$5=\overline{2}$&$6=\overline{1}$\\
					\hline
					$M_1$&0 &$\alpha_1-\alpha_2$ &$\alpha_1-\alpha_3$ &$\alpha_1+\alpha_3$ &$\alpha_1+\alpha_2$ &$2\alpha_1$\\
					\hline
					$M_2$&$\alpha_2-\alpha_1$&0&$\alpha_2-\alpha_3$&$\alpha_2+\alpha_3$&$2\alpha_2$&$\alpha_2+\alpha_1$\\
					\hline
					$M_3$&$\alpha_3-\alpha_1$ &$\alpha_3-\alpha_2$ &0 &$2\alpha_3$ &$\alpha_3+\alpha_2$ &$\alpha_3+\alpha_1$\\
					\hline
					$M_{\overline{3}}$ &$-\alpha_3-\alpha_1$ &$-\alpha_3-\alpha_2$ &$-2\alpha_3$ &0&$-\alpha_3+\alpha_2$ &$-\alpha_3+\alpha_1$\\
					\hline
					$M_{\overline{2}}$ &$-\alpha_2-\alpha_1$ &$-2\alpha_2$ &$-\alpha_2-\alpha_3$&$-\alpha_2+\alpha_3$ &0&$-\alpha_2+\alpha_1$\\	
					\hline
					$M_{\overline{1}}$	&$-2\alpha_1$&$-\alpha_1-\alpha_2$ &$-\alpha_1-\alpha_3$ &$-\alpha_1+\alpha_3$ &$-\alpha_1+\alpha_2$ &$0$\\
					\hline
					
				\end{tabular}
			\end{center}
		\end{exam}
		
		Furthermore, we define another degree 2 generator.
		\begin{defe}\label{generator_G}
			We define the function $Q:\cv_{2n-1}\to H^2(BT^n)$ by 
			\begin{equation}\label{gen_G}
				Q(j)=\alpha(j\overline{j})=-\alpha_j ~~\text{for}~~ j\in\cv_{2n-1}.
			\end{equation}		
		\end{defe}
		
		\begin{propo}
			The function $Q$ is an element of $H^2(\mathcal{GQ}_{2n-1})$.
		\end{propo}
		\begin{proof}
			The result follows since, for any edge $ij\in\ce_{2n-1}$, we have the following:
			\[Q(i)-Q(j)=-\alpha_i-(-\alpha_j)\equiv 0~~\mbox{mod}~~\alpha(ij).\]
		\end{proof}
		
		\begin{rema}
			\label{rem_submanifold_evem-dim-quadric}
			Let $\mathcal{GQ}_{2n-2}$ be the GKM subgraph of $\mathcal{GQ}_{2n-1}$ that has the same set of vertices $\cv_{2n-1}=[2n]$ and includes all edges of $\mathcal{GQ}_{2n-1}$ except \[\{v\overline{v}\in\ce_{2n-1}~:~v\in\cv_{2n-1}\}.\] 
			Combinatorially, $Q$ is the \emph{equivariant Thom class} of $\mathcal{GQ}_{2n-2}$, i.e., the normal axial functions of the GKM subgraph (see \cite[Section 4]{mmp07} for more details).
			Geometrically, this is nothing but the equivariant Thom class of the $T^{n}$-invariant submanifold $Q_{2n-2}\subset Q_{2n-1}$ defined by $z_{2n+1}=0$. In Section~\ref{sect:7}, we will study $\mathcal{GQ}_{2n-2}$ more precisely.
		\end{rema}
		
		\begin{exam} 
			For $n=2$, the following example represents $Q:\cv_3\to H^2(BT^2)$. 
			This function $Q$ is defined by the normal axial functions of the GKM subgraph $\mathcal{GQ}_2$ of $\mathcal{GQ}_3$, i.e., the equivariant Thom class of the GKM subgraph $\mathcal{GQ}_2$.
			
			\begin{figure}[H]
				\begin{tikzpicture}
					\begin{scope}[xscale=0.5, yscale=0.5]
						\coordinate (1) at (-3,-3);
						\coordinate (2) at (-3,3);
						\coordinate (3) at (3,3);
						\coordinate (4) at (3,-3);
						
						\fill(1) circle (5pt);
						\node[left] at (1) {$Q(1)=-\alpha_{1}$};
						\fill(2) circle (5pt);
						\node[left] at (2) {$Q(2)=-\alpha_{2}$};
						\fill(3) circle (5pt);
						\node[right] at (3) {$Q(4)=\alpha_{1}$}; 
						\fill(4) circle (5pt);
						\node[right] at (4) {$Q(3)=\alpha_{2}$}; 
						
						\draw[thick](1)--(2);
						
						\draw[dashed](1)--(3);
						
						\draw[thick](1)--(4);
						
						\draw[thick](3)--(2);
						
						\draw[dashed](2)--(4);
						
						\draw[thick](3)--(4);
						
					\end{scope}
				\end{tikzpicture}
				\caption{The element $Q\in H^{2}(\mathcal{GQ}_{3})$.}
				\label{G_of_lowest-example}
			\end{figure}
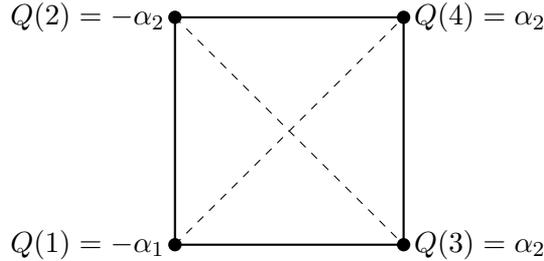
		\end{exam}
		
		\begin{exam} For $n=3$, the following table shows $Q:\cv_5\to H^2(BT^3)$. 
			This also represents the Thom class of the GKM subgraph $\mathcal{GQ}_4$ of $\mathcal{GQ}_5$ (see Figure~\ref{Thom_class_L5}).
			\begin{center}
				\begin{tabular}{|c|c|c|c|c|c|c|}
					\hline
					
					&1&2&3&$4=\overline{3}$&$5=\overline{2}$&$6=\overline{1}$\\
					\hline
					$Q$&$-\alpha_1$&$-\alpha_2$&$-\alpha_3$&$\alpha_3$&$\alpha_2$&$\alpha_1$\\
					\hline
				\end{tabular}
			\end{center}
			
			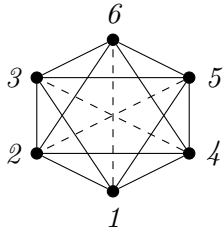
\begin{figure}[hh]
				\centering
				\begin{tikzpicture}[row sep=small, column sep=small]
					\node[draw, circle, fill=black, inner sep=1.5pt, label=right:5] (A) at (1, 0.5) {};
					\node[draw, circle, fill=black, inner sep=1.5pt, label=right:4] (B) at (1,-0.5) {};
					\node[draw, circle, fill=black, inner sep=1.5pt, label=left:3] (C) at (-1, 0.5) {};
					\node[draw, circle, fill=black, inner sep=1.5pt, label=left:2] (D) at (-1, -0.5) {};
					\node[draw, circle, fill=black, inner sep=1.5pt, label=6] (E) at (0, 1) {};	
					\node[draw, circle, fill=black, inner sep=1.5pt, label=below:1] (F) at (0, -1) {};		
					\draw (A) -- node[midway, below, sloped]  {} (B);
					\draw (B) -- node[midway, below right, sloped] {} (D);
					\draw (C) -- node[midway, below left, sloped] {} (A);
					\draw (A) -- node[midway, below, sloped]  {} (E);
					\draw (A) -- node[midway, below, sloped]  {} (F);
					\draw (B) -- node[midway, below, sloped]  {} (F);
					\draw (B) -- node[midway, below, sloped]  {} (E);
					\draw (F) -- node[midway, below, sloped]  {} (D);
					\draw (F) -- node[midway, below, sloped]  {} (C);
					\draw (D) -- node[midway, below, sloped]  {} (C);
					\draw (C) -- node[midway, below, sloped]  {} (E);
					\draw (D) -- node[midway, below, sloped]  {} (E);
					\draw [dashed](F) -- node[midway, below, sloped]  {} (E);
					\draw[dashed] (B) -- node[midway, below, sloped]  {} (C);
					\draw[dashed] (A) -- node[midway, below, sloped]  {} (D);
				\end{tikzpicture}
				\caption{The subgraph as above represents the GKM subgraph $\mathcal{GQ}_{4}\subset \mathcal{GQ}_5$.}\label{Thom_class_L5}
			\end{figure}
		\end{exam}

		\subsection{Some properties for degree 2 elements  $M_v$ and $Q$}
		In this subsection, we will prove Proposition~\ref{H(BT)-generators}.
		To do this, we will prepare two lemmas.
		
		\begin{lema}
			\label{rel-M_v-and-Q}
			For any $v\in\cv_{2n-1}$, 
			the following equality holds:
			\begin{equation}\label{3.11}
				M_v+M_{\overline{v}}=2Q.
			\end{equation}
		\end{lema}
		\begin{proof}
			Let $w\in \cv_{2n-1}$ be a vertex of $\mathcal{GQ}_{2n-1}$. 
			We now verify the result by evaluating \eqref{3.11} for each $w\in \cv_{2n-1}$.
			\begin{enumerate}[(a)]
				\item \underline{If $w=v$}: $\big(M_v+M_{\overline{v}}\big)(w)=0+2\alpha_{\overline{v}}=-2\alpha_v=2Q(w).$
				\item \underline{If $w=\overline{v}$}: $\big(M_v+M_{\overline{v}}\big)(w)=2\alpha_{v}+0=-2\alpha_{\overline{v}}=2Q(w).$
				\item \underline{If $w\not=v,\overline{v}$}: $\big(M_v+M_{\overline{v}}\big)(w)=-\alpha_w+\alpha_v+(-\alpha_w+\alpha_{\overline{v}})=-2\alpha_{w}=2Q(w)$.
			\end{enumerate}
		\end{proof}
		
		\begin{lema}
			\label{2alpha_generating}
			For any $v\in\cv_{2n-1}$, 	the following equality holds:
			\begin{equation}\label{3.12} 
				M_v-M_{\overline{v}}=2\alpha_v.
			\end{equation}
		\end{lema}
		\begin{proof}  We verify the result by evaluating \eqref{3.12} for each $w\in \cv_{2n-1}$.
			\begin{enumerate}
				\item \underline{If $w=v$}: $\big(M_v-M_{\overline{v}}\big)(w)=-2\alpha_{\overline{v}}=2\alpha_v.$	
				\item \underline{If $w=\overline{v}$}: $\big(M_v-M_{\overline{v}}\big)(w)=2\alpha_{v}.$
				\item \underline{If $w\not=v,\overline{v}$}: $\big(M_v-M_{\overline{v}}\big)(w)=-\alpha_w+\alpha_v-(-\alpha_w+\alpha_{\overline{v}})=2\alpha_{v}.$
			\end{enumerate}
		\end{proof}
		
		\begin{propo}
			\label{H(BT)-generators}
			The generator $\alpha_i\in H^*(BT^n)$, for $i=1,\ldots,n$, is obtained by the following equality:
			\begin{equation}\label{3.13}
				\alpha_i=M_i-Q.
			\end{equation}
		\end{propo}
		\begin{proof}
			The proposition follows from \eqref{3.11} and \eqref{3.12}.
		\end{proof}

		\subsection{Higher degree generators}\label{sec:high_deg_gen}
		
		In this subsection, we define a degree-$2m$ element $\Delta_K \in H^{2m}(\mathcal{GQ}_{2n-1})$ for some $K \subset \cv_{2n-1}$ with $|K| = 2n - m$. Later, in Lemma~\ref{surjmainth}, we prove that these elements, together with $Q$ and $M_v$ for $v \in \cv_{2n-1}$, generate $H^{*}(\mathcal{GQ}_{2n-1})$.
		
		\begin{defe}\label{Delta_k}
			Let $K\subset \cv_{2n-1}=[2n]$ be a non-empty subset such that 
			$\{i, \overline{i}\}\not\subset K$ for every $i\in \mathcal{V}_{2n-1}$.
			We define the function $\Delta_K :\cv_{2n-1}\to H^{4n-2|K|}(BT^n)$ by
			\begin{equation}\label{delta_keqn}
				\Delta_K(j)=\begin{cases}
					\prod_{k\notin K}\alpha(jk) &~~\text{if}~~ j\in K,\\
					0 &~~\text{if}~~j\notin K.
				\end{cases}
			\end{equation}
			Note that, by definition, $\Delta_{\emptyset}=0$.
		\end{defe} 
		
		\begin{lema}
			If $K\subset\cv_{2n-1}$ is a non-empty subset such that $\{i, \overline{i}\}\not\subset K$ for every $i\in \mathcal{V}_{2n-1}$, then $\Delta_K\in H^{4n-2|K|}(\mathcal{GQ}_{2n-1})$.
		\end{lema}
		\begin{proof}
			Let $\Gamma_K$ be the full subgraph spanned by $K$ such that $\{i, \overline{i}\}\not\subset K$ for every $i\in \mathcal{V}_{2n-1}$, and let $\mathcal{GK}$ be the pair consisting of $\Gamma_{K}$ and the restricted axial function $\alpha_{K}:=\alpha|_{\mathcal{E}(\Gamma_{K})}$ on $\mathcal{E}(\Gamma_{K})$. 
			Since $\mathcal{GK}$ is closed under the connection $\nabla$ (see \eqref{conn}), $\mathcal{GK}$ is a GKM subgraph of $\mathcal{GQ}_{2n-1}$  (see Definition \ref{gkmsubgraph}). Moreover, $\Delta_{K}$ is the \emph{equivariant Thom class} of the GKM subgraph $\mathcal{GK}$ (see \cite[Section 4]{mmp07}). Therefore, by similar arguments as in \cite[Lemma 4.1]{mmp07}, the lemma follows.
		\end{proof}
		
		\begin{exam}
			For the GKM graph $\mathcal{GQ}_5$, the set of vertices $P=\{4,5,6\}$ (see Figure \ref{xyz} below) satisfies the condition given in Definition \ref{Delta_k}. 
			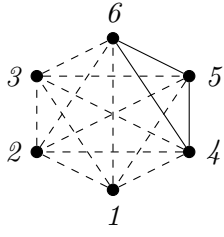
\begin{figure}[H]
				\centering
				\begin{tikzpicture}
					\node[draw, circle, fill=black, inner sep=1.5pt, label=right:5] (A) at (1, 0.5) {};
					\node[draw, circle, fill=black, inner sep=1.5pt, label=right:4] (B) at (1,-0.5) {};
					\node[draw, circle, fill=black, inner sep=1.5pt, label=left:3] (C) at (-1, 0.5) {};
					\node[draw, circle, fill=black, inner sep=1.5pt, label=left:2] (D) at (-1, -0.5) {};
					\node[draw, circle, fill=black, inner sep=1.5pt, label=6] (E) at (0, 1) {};	
					\node[draw, circle, fill=black, inner sep=1.5pt, label=below:1] (F) at (0, -1) {};		
					\draw (A) -- node[midway, below, sloped]  {} (B);
					\draw [dashed] (B) -- node[midway, below right, sloped] {} (D);
					\draw [dashed] (C) -- node[midway, below left, sloped] {} (A);
					\draw (A) -- node[midway, below, sloped]  {} (E);
					\draw[dashed]  (A) -- node[midway, below, sloped]  {} (F);
					\draw[dashed] (B) -- node[midway, below, sloped]  {} (F);
					\draw  (B) -- node[midway, below, sloped]  {} (E);
					\draw[dashed]  (F) -- node[midway, below, sloped]  {} (D);
					\draw[dashed]  (F) -- node[midway, below, sloped]  {} (C);
					\draw[dashed]  (D) -- node[midway, below, sloped]  {} (C);
					\draw [dashed] (C) -- node[midway, below, sloped]  {} (E);
					\draw[dashed]  (D) -- node[midway, below, sloped]  {} (E);
					\draw [dashed](F) -- node[midway, below, sloped]  {} (E);
					\draw[dashed] (B) -- node[midway, below, sloped]  {} (C);
					\draw[dashed] (A) -- node[midway, below, sloped]  {} (D);
				\end{tikzpicture}
				\caption{ $\Gamma_{P}$ : a subgraph consisting of the vertices $P=\{4,5,6\}\subset{\cv_5}$.}
				\label{xyz}
			\end{figure}
			Therefore, according to Table~\ref{table1} in Section~\ref{sect:2.2}, the function $\Delta_{P}$ for the subgraph $\Gamma_{P}$ is defined by 
			\[\Delta_{P}(i)=\begin{cases}
				0 ~~&\text{if}~~i=1,2,3,\\
				\alpha(41)\alpha(42)\alpha(43)= 
				\alpha_3(\alpha_1+\alpha_3)(\alpha_2+\alpha_3)~~&\text{if}~~i=4,\\
				\alpha(51)\alpha(52)\alpha(53)=\alpha_2(\alpha_1+\alpha_2)(\alpha_2+\alpha_3)~~&\text{if}~~i=5,\\
				\alpha(61)\alpha(62)\alpha(63)=
				\alpha_1(\alpha_1+\alpha_2)(\alpha_1+\alpha_3)~~&\text{if}~~i=6.
			\end{cases}\]	
		\end{exam}

		\begin{rema}
			\label{rel-M-delta_n1}
			For $n=1$, i.e., $\mathcal{GQ}_{1}$, its vertices are defined by $\mathcal{V}_{1}:=\{1,\overline{1}\}$.
			In this case, from the definitions of $M_{a}$ and $\Delta_{a}$ for $a\in \mathcal{V}_{1}:=\{1,\overline{1}\}$, we obtain the following equality:
			\begin{align*}
				M_{\overline{a}}=2\Delta_{a}.
			\end{align*}
			Moreover, we have 
			\begin{align*}
				Q=\Delta_{1}+\Delta_{\overline{1}}.
			\end{align*}
			Therefore, for $n=1$, the generators can be reduced to $\Delta_{1}, \Delta_{\overline{1}}\in H^{2}(\mathcal{GQ}_{1})$.
			Furthermore, since $\mathcal{GQ}_{1}$ is the same GKM graph induced from the standard $T^{1}$-action on $\mathbb{CP}^{1}(\simeq Q_{1})$,
			it follows from a well-known formula that we have 
			\begin{align*}
				H^{*}_{T^{1}}(Q_{1})\simeq H^{*}(\mathcal{GQ}_{1})\simeq \mathbb{Z}[\Delta_{1}, \Delta_{\overline{1}}]/\langle \Delta_{1}\Delta_{\overline{1}} \rangle\simeq H^{*}_{T^{1}}(\mathbb{CP}^{1}).
			\end{align*}
		\end{rema}

		\section{Five types of relations among $M_v$, $Q$ and $\Delta_K$'s}
		\label{rel}
		This section introduces the five types of relations among  $M_v, ~Q$ and $\Delta_K$'s.
		
		We use the following notation for $J=\cv_{2n-1}\setminus\{v\}$ or, when $J$ satisfies the property that $\{i,\overline{i}\}\not\subset J$ for all $i\in\cv_{2n-1}$ with $n\ge 2$:
		\begin{equation}\label{G_J}
			G_J=\begin{cases}
				M_v &\text{if $J=\cv_{2n-1}\setminus\{v\}$}\\
				\Delta_J&\text{if $J$ satisfies the property that $\{i,\overline{i}\}\not\subset J$ for all $i\in\cv_{2n-1}$.}
		\end{cases}\end{equation}
		Note that the \emph{support} of class $G_J$ is precisely $J$ i.e., \[\{i\in\cv_{2n-1}~:~G_J(i)\neq 0\}=J.\] Then we can write the first multiplicative relation as follows.
		\begin{lema}[Relation 1]\label{lem-relation1}
			For $n\ge 2$, let $\mathcal{J}$ be a  subset of the following set 
			\[ \Big\{J\subset \cv_{2n-1} :\text{$\{i,\overline{i}\}\not\subset J$ for all $i\in\cv_{2n-1}$}\Big\}\bigcup \Big\{\cv_{2n-1}\setminus\{v\}:v\in\cv_{2n-1}\Big\}.\]
			If $\displaystyle\bigcap_{J\in\mathcal{J}}J=\emptyset$, then the following relation holds: 
			\begin{equation}
				\prod_{ J\in \mathcal{J}} G_J=0.
			\end{equation}
			For $n=1$, i.e., $\mathcal{V}_{1}=\{1,\overline{1}\}$, there are the following relations:
			\begin{equation*}
				M_{1}M_{\overline{1}}=M_{1}\Delta_{1}=M_{\overline{1}}\Delta_{\overline{1}}=\Delta_{1}\Delta_{\overline{1}}=0.
			\end{equation*}
		\end{lema}
		\begin{proof}
			The lemma follows from the Definitions \ref{M_v} and \ref{Delta_k}.
		\end{proof}
		
		The next relation follows straightforwardly from Lemma~\ref{rel-M_v-and-Q}.
		\begin{lema}[Relation 2]\label{relation2}
			For every distinct $i,j\in\cv_{2n-1}$, the following relation holds:
			\[M_i+M_{\overline{i}}=M_j+M_{\overline{j}}~=2Q.\]
		\end{lema}
		\begin{proof}
			The lemma follows from \eqref{3.11}. 
		\end{proof}
		
		The next relation follows straightforwardly from Definition \ref{M_v} and Definition~\ref{Delta_k}.
		\begin{lema}[Relation 3]
			\label{lemma-relation3}
			If $K\subset \mathcal{V}_{2n-1}$ satisfies that $|K|=n$ and $\{i,\overline{i}\}\not\subset K$ for every $i\in \mathcal{V}_{2n-1}$, then the following relation holds:
			\begin{equation}
				\label{rel3-delta-M}
				2\Delta_{K}=\prod_{{i\in K^{c}}} M_i.
			\end{equation}
		\end{lema}
		
		\begin{proof} 
			Note that both \(2\Delta_{K}\) and \(\prod_{i\in K^{c}} M_i\) are nonzero only on the elements in \(K\).  
			If \(j\in K\) (i.e., \(\overline{j}\in K^c\)), then by \eqref{delta_keqn} we have  
			\[
			2\Delta_K(j)
			= 2\prod_{i\notin K}\alpha(ji)
			= 2(-\alpha_j)\prod_{\substack{i\notin K\\ i\neq \overline{j}}}(-\alpha_j+\alpha_i).
			\]
			On the other hand, by \eqref{mapmv}
			\[
			\prod_{i\in K^{c}} M_i(j)
			= M_{\overline{j}}(j)\cdot \prod_{\substack{i\in K^{c}\\ i\neq \overline{j}}} M_i(j)
			= -2\alpha_j \cdot \prod_{\substack{i\in K^{c}\\ i\neq \overline{j}}} (-\alpha_j+\alpha_i).
			\]
			Hence, the lemma follows.
		\end{proof}
		
		\begin{rema}\label{rema_for_rel3}
			For $n=1$, i.e., $\mathcal{V}_{1}=\{1,\overline{1}\}$,  \eqref{rel3-delta-M} yields the following three relations: 
			\begin{equation*}
				2\Delta_{1}=M_{\overline{1}},\quad  2\Delta_{\overline{1}}=M_{1},\quad 2\Delta_{\emptyset}=0=M_{1}M_{\overline{1}}.
			\end{equation*}
			Moreover, together with Lemma~\ref{relation2} (Relation 2), we have
			\begin{equation*}
				\Delta_{1}+\Delta_{\overline{1}}=Q.
			\end{equation*}
		\end{rema}
		\begin{exam}\label{exple_for_rel3}
			For $n=2$ and $K=\{1,2\}\subset \cv_3$, we have 
			\[2\Delta_{K}=M_{\overline{1}}M_{\overline{2}}\in H^4(\mathcal{GQ}_3).\]
		\end{exam}

		We also have the following multiplicative relation for two generators, $\Delta_K$ and $M_i$:
		\begin{lema}[Relation 4]\label{lem-relation4}
			Fix $i\in \cv_{2n-1}$. If $K\subset\cv_{2n-1}$ satisfies $\{i\}\subsetneq K$ and  $\{v,\overline{v}\}\not\subset K$ for every $v\in\cv_{2n-1}$, then the following equality holds: \[\Delta_K\cdot M_i=\Delta_{K\setminus\{i\}}.\]
		\end{lema}
		\begin{proof}
			Note that $\Delta_K\cdot M_i$ is non-zero only on $K\cap (\cv_{2n-1}\setminus\{i\})=K\setminus\{i\}$. 
			By the assumption of $K$, we have $\overline{i}\not\in K\setminus\{i\}$. 
			This implies that we can define the element $\Delta_{K\setminus\{i\}}$ by Definition~\ref{Delta_k}.
			On the other hand, since $\overline{i}\not\in K\setminus\{i\}$, and using the definitions of $\Delta_{K}$ and $M_{i}$ (see Definition \ref{Delta_k} and Definition \ref{M_v}),
			we have the following equality:
			\begin{equation}
				\big(\Delta_K\cdot M_i\big)(w)=\begin{cases}
					\displaystyle	\prod_{k\notin K\setminus\{i\}}\alpha(wk)~~&\text{if}~~w\in K\setminus\{i\}\\
					0~~&\text{if}~~w\notin K\setminus\{i\}.
				\end{cases}
			\end{equation}
			Hence, from \eqref{delta_keqn}, the lemma follows. \end{proof}
		
		\begin{rema}\label{rema_for_rel4}
			For $n=1$, i.e., $\mathcal{V}_{1}=\{1,\overline{1}\}$, this relation yields the following two relations: 
			\begin{equation*}
				\Delta_{1}M_{1}=\Delta_{\overline{1}}M_{\overline{1}}=\Delta_{\emptyset}=0.
			\end{equation*}
		\end{rema}
		
		\begin{exam}
			For $n=3$, $i=1$ and $K=\{1,3,5\}$, we have \[\Delta_{\{1,3,5\}}\cdot M_1=\Delta_{\{3,5\}}.\]
		\end{exam}

	
		If $L\subset \mathcal{V}_{2n-1}$ satisfies that $|L|=n-1$ and $\{i,\overline{i}\}\not\subset L$ for every $i\in \mathcal{V}_{2n-1}$. Then because $|\mathcal{V}_{2n-1}|=2n$, there exists the unique pair $\{a,\overline{a}\}\subset L^c=\cv_{2n-1}\setminus L$. Consequently, we have 
	\[\Delta_{L^c\setminus \{a\}}, \Delta_{L^c\setminus \{\overline{a}\}}\in H^{2n}(\mathcal{GQ}_{2n-1}).\]
	We therefore obtain the following lemma.

	\begin{lema}[Relation 5]	\label{lemma-relation5errata}
		For every $L\subset \mathcal{V}_{2n-1}$ as above, we have 
		\[
		\Delta_{L^c\setminus\{a\}}
		+
		\Delta_{L^c\setminus\{\overline{a}\}}
		=
		Q\prod_{i\in L}M_i.
		\]
	\end{lema}
	
	\begin{proof}
		It suffices to verify the equality at every vertex
		\(v\in\mathcal{V}_{2n-1}\).
		
		First, suppose that \(v\in L\). Then
		\[
		\Delta_{L^c\setminus\{a\}}(v)
		=
		\Delta_{L^c\setminus\{\overline{a}\}}(v)
		=0,
		\]
		since \(v\in L\) occurs among the factors defining each of the two
		classes. Moreover,
		\[
		\left(Q\prod_{i\in L}M_i\right)(v)=0,
		\]
		because the factor \(M_v(v)=\alpha(vv)\) vanishes. Thus the desired
		identity holds at \(v\).
		
		Next, let \(v=a\). Since
		\[
		\mathcal{V}_{2n-1}\setminus
		\bigl(L^c\setminus\{a\}\bigr)
		=
		L\cup\{a\},
		\]
		we have
		\[
		\Delta_{L^c\setminus\{a\}}(a)=0.
		\]
		On the other hand,
		\[
		\begin{aligned}
			\Delta_{L^c\setminus\{\overline{a}\}}(a)=
			\prod_{j\notin L^c\setminus\{\overline{a}\}}\alpha(aj)
			=
			\alpha(a\overline{a})\prod_{j\in L}\alpha(aj) =
			Q\prod_{j\in L}M_j(a),
		\end{aligned}
		\]
		Hence
		\[
		\Delta_{L^c\setminus\{a\}}(a)
		+
		\Delta_{L^c\setminus\{\overline{a}\}}(a)
		=
		\left(Q\prod_{j\in L}M_j\right)(a).
		\]
		
		The case \(v=\overline{a}\) is analogous.
		
		It remains to consider $v\in L^c\setminus\{a,\overline{a}\}$. Since \(L^c\) contains the unique  pair
		\(\{a,\overline{a}\}\), it follows that \(\overline{v}\in L\).  
		Hence,
		\[
		\begin{aligned}
			\Delta_{L^c\setminus\{a\}}(v)+
			\Delta_{L^c\setminus\{\overline{a}\}}(v)&=
			\prod_{j\notin L^c\setminus\{a\}}\alpha(vj)
			+
			\prod_{j\notin L^c\setminus\{\overline{a}\}}\alpha(vj)= 	\bigl(\alpha(va)+\alpha(v\overline{a})\bigr)
			\prod_{j\in L}\alpha(vj)\\
			&=-2\alpha_v \cdot (-\alpha_v)	\prod_{j\in {L\setminus\{\overline{v}\}}}(-\alpha_v+\alpha_j)=\left (Q M_{\overline{v}}\prod_{j\in L\setminus\{\overline{v}\}}M_j \right ) (v).
		\end{aligned}
		\]
		Thus the two sides agree at every vertex
		\(v\in\mathcal{V}_{2n-1}\), and hence
		\[\Delta_{L^c\setminus\{a\}}+\Delta_{L^c\setminus\{\overline{a}\}}=Q\prod_{i\in L}M_i.	\]
		This completes the proof.
	\end{proof}

			\section{Main theorem and its proof}
			\label{secmainth} 
			
			This section aims to prove this paper's main theorem (Theorem \ref{mainth}). To do that, we first prepare some notations. Then, in Lemma  \ref{surjmainth}, we show that $M_v$, $Q$ and $\Delta_K$ as defined in Section \ref{gen}, generate $H^*(\mathcal{GQ}_{2n-1})$ as a ring. In the final step, we will prove that the relations defined in Section \ref{rel} are enough relations to describe the ring structure of $H^*(\mathcal{GQ}_{2n-1})$. As a consequence, we present the equivariant cohomology of $Q_{2n-1}$ (i.e., $H^*_{T^{n}}(Q_{2n-1})$) in terms of generators and relations. 
			
			We first prepare some notations.
			Let $\mathbf{M}$ denote the following set of cohomology classes in $H^{*}(\mathcal{GQ}_{2n-1})$:
			\[\{M_v~:~v\in\cv_{2n-1}\},\] and, 
			let $\mathbf{D}$ denote  the set of cohomology classes  \[\{\Delta_P~:~\text{$P(\not=\emptyset)\subset\cv_{2n-1}, ~\{i,\overline{i}\}\not\subset P$ for all $i\in\cv_{2n-1}$\}}.\]
			Let $\bz[\mathbf{M}, \mathbf{D}, Q]$ be the polynomial ring generated by all elements in $\mathbf{M}, \mathbf{D}$ and $Q$.
			
			Let $\mathfrak{I}\lhd \bz[\mathbf{M}, \mathbf{D}, Q]$ be an ideal generated by the following types of elements:
			\begin{enumerate}[(i)]
				\item $\displaystyle\prod_{\cap J=\emptyset} G_J$ for $G_J$ as in \eqref{G_J} if $n\ge 2$. If $n=1$, $M_{1}M_{\overline{1}},~ \Delta_{1}M_{1},~ \Delta_{\overline{1}}M_{\overline{1}}$ and $\Delta_{1}\Delta_{\overline{1}}$;
				\item $M_v+M_{\overline{v}}-2Q$ for every $v\in\cv_{2n-1}$;
				\item $\displaystyle2\Delta_{K}-\prod_{i\in K^{c}}M_{i}$ for every $K\subset \mathcal{V}_{2n-1}$ such that $|K|=n$ and $\{i,\overline{i}\}\not\subset K$ for all $i\in \mathcal{V}_{2n-1}$;
				\item $\Delta_P\cdot M_i-\Delta_{P\setminus\{i\}}$ for every subset $P\subset \cv_{2n-1}$ such that {$\{i\}\subset P$} 
				and  $\{j,\overline{j}\}\not\subset P$ for all $j\in\cv_{2n-1}$.
				\item $\Delta_{L^c\setminus\{a\}}+	\Delta_{L^c\setminus\{\overline{a}\}}-Q\prod_{i\in L}M_i$ for every  $L\subset \mathcal{V}_{2n-1}$ such that  $|L|=n-1$,  $\{i,\overline{i}\}\not\subset L$ for every $i\in \mathcal{V}_{2n-1}$ and $\{a,\overline{a}\}\subset L^c=\cv_{2n-1}\setminus L$ is  the unique pair. 
			\end{enumerate}
			We define \[\bz[\mathcal{GQ}_{2n-1}]:= \bz[\mathbf{M}, \mathbf{D}, Q]/\mathfrak{I}.\] Let $\widetilde{\varphi}:\bz[\mathbf{M}, \mathbf{D}, Q]\to H^*(\mathcal{GQ}_{2n-1})$ be a ring homomorphism, and let  
			\begin{equation}
				\label{phi}
				{\phi}:\bz[\mathcal{GQ}_{2n-1}]\to H^*(\mathcal{GQ}_{2n-1})
			\end{equation} 
			be the homomorphism induced from $\widetilde{\varphi}$.
			Note that we can immediately prove the well-definedness of $\varphi$ by lemmas in Section~\ref{rel}.
			
			In other words, the following diagram commutes:
			\begin{equation}\label{commdiag}
				\begin{tikzcd}
					&\bz[\mathbf{M}, \mathbf{D}, Q]\arrow[d]	\arrow[dr, "\widetilde{\phi}"] &\\
					&\bz[\mathcal{GQ}_{2n-1}] \arrow[r, "\phi"]&H^*(\mathcal{GQ}_{2n-1})
				\end{tikzcd}
			\end{equation}
			where the vertical map is the natural projection.
			Now we state the main theorem of this paper.
			\begin{guess}
				\label{mainth}
				The homomorphism $\varphi$ is an isomorphism of rings, i.e.,
				\begin{equation}\label{mainth1}
					\varphi: \bz[\mathcal{GQ}_{2n-1}]\xrightarrow{\cong} H^*(\mathcal{GQ}_{2n-1}).\end{equation} 
				In particular, for the effective $T^{n}$-action on $Q_{2n-1}$ defined by \eqref{introtorusaction}, the following isomorphism holds:	
				\begin{equation}
					\label{mainth2}
					H^*_{T^{n}}(Q_{2n-1})\cong H^*(\mathcal{GQ}_{2n-1})\cong  \bz[\mathcal{GQ}_{2n-1}].
				\end{equation}
			\end{guess}
			In the proofs below, for simplicity, we often abuse the notations $M_i$ and $\Delta_P$ to mean $\widetilde{\varphi}(M_i)$ and $\widetilde{\varphi}(\Delta_P)$, respectively.
			We first prove the surjectivity of $\varphi$.
			
			\begin{lema}\label{surjmainth}
				The homomorphism $\varphi: \bz[\mathcal{GQ}_{2n-1}]\to H^*(\mathcal{GQ}_{2n-1})$ is surjective.
			\end{lema}
			\begin{proof}
				This is enough to prove that $\widetilde{\varphi}:\bz[\mathbf{M}, \mathbf{D}, Q] \to H^*(\mathcal{GQ}_{2n-1})$ is surjective. 
				Consider an element $f\in H^*(\mathcal{GQ}_{2n-1})$. 
				For the vertex  $1\in\cv_{2n-1}$, one can write $f(1)\in H^*(BT^{n})\cong\bz[\alpha_1,\ldots,\alpha_n]$  (see \eqref{congrel}) as
				\[f(1)=\sum_{\mathbf{j}}c_{\mathbf{j}}\alpha_1^{j_1}\cdots \alpha_{n}^{j_{n}}={h_1}\] 
				where $c_{\mathbf{j}}\in\bz$ and $\mathbf{j}=(j_1,\ldots, j_{n})\in (\mathbb{N}\cup\{0\})^{n}$. 
				Since $\alpha_i=M_i-Q$ (see \eqref{3.13}) for each $i$,  it follows that 
				\[f(1)=\sum_{\mathbf{j}}c_{\mathbf{j}}(M_1-Q)^{j_1}\cdots (M_n-Q)^{j_{n}}(1)=h_1.\]
				This implies that there exists an element in $\bz[\mathbf{M}, \mathbf{D}, Q]$ whose image under $\widetilde{\varphi}$ coincides with $f(1)$ on the vertex $1\in \cv_{2n-1}$.
				
				Considering $h_1\in H^*(BT^n)$ as a constant map (cf. \eqref{eqn3.10}) in  $H^*(\mathcal{GQ}_{2n-1})$, we next define	  $f_2:=f-h_1$ and hence $f_2(1)=0$. By using the congruence relation \eqref{congrel} on the edge $21\in\ce_{2n-1}$, we have
				\[f_2(2)-f_2(1)\equiv0~~~\text{mod}~\alpha(21)=M_1(2).\]
				Therefore, $f_2(2)=h_2M_1(2)$ for some $h_2\in H^*(BT^{n})$.
				
				Furthermore, from \eqref{3.13}, we have that $h_2M_1$ is in the image of $\widetilde{\varphi}$. 
				Define \[f_3=f_2-h_2M_1~(=f-h_1-h_2M_1).\]
				Since $f_2(1)=0$, $M_1(1)=0$, and $f_2(2)=h_2M_1(2)$, we have
				\[f_3(1)=0=f_3(2).\] 
				Thus, by the congruence relations on the edges $31$ and $32\in\ce_{2n-1}$, we can  write
				\[f_3(3)=h_3M_1M_2(3)\] for some $h_3\in H^*(BT^n)\subset \bz[\mathbf{M}, \mathbf{D}, Q]$.
				
				Similarly, we can also verify that $f_4:=f_3-h_3M_1M_2$ satisfies $f_4(1)=f_4(2)=f_4(3)=0$.	
				By iterating this procedure $n$ times, we obtain an element
				\[f_{n}:=f_{n-1}-h_{n-1}M_1\cdots M_{n-2}\] for some $h_{n-1}\in H^*(BT^n)\subset \bz[\mathbf{M}, \mathbf{D}, Q]$ and $f_{n}\in H^*(\mathcal{GQ}_{2n-1})$, which satisfies \[f_{n}(i)=0 \text{ for $1\leq i\leq n-1$}\] as $f_{n-1}(i)=0$ for $1\leq i\leq n-2$ and $f_{n-1}(n-1)=h_{n-1}M_1\cdots M_{n-2}(n-1)$. 
				
				Furthermore, applying the congruence relations for the edges $ni\in\ce_{2n-1}$ for $1\leq i\leq n-1$, we have 
				\[f_n(n)=h_n M_1\cdots M_{n-1}(n)\]
				for some $h_n\in H^*(BT^n)$.
				Then, we have
				\begin{equation}
					\begin{split}\label{surj1}
						f_{n+1}&=f_{n}-h_{n}M_1\cdots M_{n-1}\\
						&=f_{n-1}-h_{n-1}M_1\cdots M_{n-2}-h_{n}M_1\cdots M_{n-1}\\
						&~~~~~~~~\vdots\\
						&=f-h_1-h_2M_1-\cdots-h_{n-1}M_1\cdots M_{n-2}-h_{n}M_1\cdots M_{n-1}
					\end{split}
				\end{equation}
				which statisfies $f_{n+1}(i)=0$ for all $1\leq i\leq n$.
				Therefore by using the congruence relation \eqref{congrel} for the edge $\overline{n}i\in\ce_{2n-1}$ for every $1\leq i\leq n$ (where $\overline{n}=n+1$), 
				one can verify that 
				\[f_{n+1}(n+1)\equiv 0~~~\text{mod}~\alpha(\overline{n}i)\]
				for all $1\leq i\leq n$. 
				In particular, from the definition of $\Delta_{\{n+1,\ldots,2n\}}$ (see Definition \ref{Delta_k}), one can write   
				\[f_{n+1}(n+1)=h_{n+1}\Delta_{\{n+1,\ldots,2n\}}(n+1),\]
				for some $h_{n+1}\in  H^*(BT^{n})\subset \bz[\mathbf{M}, \mathbf{D}, Q]$. 
				
				Next, we define \[f_{n+2}=f_{n+1}-h_{n+1}\Delta_{\{n+1,\ldots,2n\}}\] which satisfies $f_{n+2}(1)=\cdots =f_{n+2}(n+1)=0$ since $\Delta_{\{n+1,\ldots,2n\}}(i)=0$ for all $1\leq i\leq n$.
				
				Similarly, for  $2\leq j\leq n$, there exists $h_{n+j}\in  H^*(BT^n)\subset \bz[\mathbf{M}, \mathbf{D}, Q]$ such that
				\begin{equation}\label{6.17}
					f_{n+j+1}:=f_{n+j}-h_{n+j}\Delta_{\{n+j,\ldots,2n\}}
				\end{equation}
				satisfies $f_{n+j+1}(1)=\cdots=f_{n+j+1}(n+j)=0$.
				Notice that, when $j=n$, \[f_{2n+1}:=f_{2n}-h_{2n}\Delta_{\{2n\}},\] which satisfies   $f_{2n+1}(1)=\cdots=f_{2n+1}(2n)=0$, i.e., $f_{2n+1}\equiv 0$.
				Therefore, $f_{2n}=h_{2n}\Delta_{\{2n\}}$. 
				By using \eqref{6.17} repeatedly, we have the following:
				\begin{equation}\label{surj2}
					\begin{split}
						f_{2n-1}&=h_{2n-1}\Delta_{\{2n-1,2n\}}+h_{2n}\Delta_{\{2n\}};\\
						f_{2n-2}&=h_{2n-2}\Delta_{\{2n-2,2n-1,2n\}}+h_{2n-1}\Delta_{\{2n-1,2n\}}+h_{2n}\Delta_{\{2n\}};\\
						&~~~~\vdots\\
						f_{n+1}&=h_{n+1}\Delta_{\{n+1,\ldots,2n\}}+\cdots+h_{2n-2}\Delta_{\{2n-2,2n-1,2n\}}+h_{2n-1}\Delta_{\{2n-1,2n\}}+h_{2n}\Delta_{\{2n\}}.
					\end{split}
				\end{equation}
				Therefore, from \eqref{surj1} and \eqref{surj2}, we have 
				\begin{equation}\label{formatcohomclass}
					\begin{split}
						f=h_1+h_2M_1&+\cdots+h_{n}M_1\cdots M_{n-1}\\&+h_{n+1}\Delta_{\{n+1,\ldots,2n\}}+\cdots+h_{2n-2}\Delta_{\{2n-2,2n-1,2n\}}+h_{2n-1}\Delta_{\{2n-1,2n\}}+h_{2n}\Delta_{\{2n\}}
					\end{split}
				\end{equation}
				where $h_i\in H^*(BT^n)\subset\bz[\mathbf{M}, \mathbf{D}, Q]$ for each $1\leq i\leq 2n$.
				Hence, the lemma follows.
			\end{proof}
			
%
%
%
%
%


We  identify $H^*(BT^n)$ with the subring $A:=Z[M_{1}-Q,\ldots, M_{n}-Q]\subset \bz[\mathcal{GQ}_{2n-1}]$ (see, \eqref{3.13}).
\begin{lema}\label{surjremark1}
				Assume $n\ge 2$. 
				Let $B_{r}:=\{n+r+1,\ldots, 2n\}$ for $0\le r\le n-1$.
				Then, every element $f\in \mathbb{Z}[\mathcal{GQ}_{2n-1}]$ can be written in the form
				\begin{equation}\label{key-equiation}
					f=h_1+h_2M_1+\cdots+h_{n}M_1\cdots M_{n-1}+h_{n+1}\Delta_{B_{0}}+\cdots +h_{2n}\Delta_{B_{n-1}}+\mathfrak{I},
				\end{equation}
				where $h_{i}$'s are some elements in $A$.
			\end{lema} 
			
		\begin{proof}
				We first reduce the generators in $\mathbf{D}$.
				Let $P\subset \mathcal{V}_{2n-1}$ be a subset such that the following condition $(\ast)$
				\[
				\{i,\bar{i}\}\not\subset P 
				\]
				for every $i=1,\ldots, n$.
				If $|P|<n$, there exists $L\supset P$ which satisfies $|L|=n$ and $(\ast)$.
				By applying Relation (iv) repeatedly, we obtain that
				\begin{align}
					\label{low_high}
					\Delta_{P}=\Delta_{L}\prod_{i\in L\setminus P} M_{i}.
				\end{align}
				Hence it is enough to consider the class $\Delta_{L}$ with $|L|=n$.
				
				Next we show that a product of two such classes can be reduced to a polynomial multiple of one of them. 
				Let $L, H$ be subsets which satisfies $|L|=|H|=n$ and $(\ast)$.
				There exists a sequence
				\[
				H=H_{0},\ H_{1}, \cdots, H_{s}=L^{c}
				\]
				of $H_{r}$ with $|H_{r}|=n$ and $(\ast)$ for $0\le r\le s$, where  
				$H_{r+1}$ is obtained from $H_{r}$ by replacing one vertex $j$ with $\bar{j}$. 
				By Relation (v), for each r, we have that 
				\[
				\Delta_{H_{r}}+\Delta_{H_{r+1}}=Qf_{r}(M),
				\]
				for some monomial $f_{r}(M)$ in $\mathbb{Z}[\mathbf{M}]$.
				Multiplying by $\Delta_{L}$, we have 
				\[
				\Delta_{L}\Delta_{H_{r}}=Qf_{r}(M)\Delta_{L}-\Delta_{L}\Delta_{H_{r+1}}
				\]
				Iterating this equality and using $L\cap L^{c}=\emptyset$, Relation (i) gives 
				$\Delta_{L}\Delta_{L^{c}}=0$.
				Therefore, 
				\begin{align}\label{produt_reducing}
					\Delta_{L}\Delta_{H}\in \mathbb{Z}[\mathbf{M},Q]\Delta_{L}
				\end{align}
				Applying this argument repeatedly shows that every product of two or more elements in $\mathbf{D}$ can be reduced to a product of a polynomial in $\mathbb{Z}[\mathbf{M}, Q]$ and a single element $\Delta_{L}$ with $|L|=n$.
				
				Since every $L$ with $|L|=n$ and $(\ast)$ can connect to $B_{0}=\{\bar{1},\ldots, \bar{n}\}$ by successive replacements $j$ and $\bar{j}$, Relation (v) implies
				\begin{align}\label{reducing to B0}
					\Delta_{L}\in \mathbb{Z}[\mathbf{M}, Q]+\mathbb{Z}[\mathbf{M}, Q]\Delta_{B_{0}}.
				\end{align}
				
				By Relation (ii), $M_{\bar{i}}=2Q-M_{i}$,
				and hence every polynomial in $\mathbb{Z}[\mathbf{M},Q]$
				can be written as a polynomial in 
				\[
				A[Q]=\mathbb{Z}[Q, M_{1}-Q,\ldots, M_{n}-Q].
				\]
				Therefore, \eqref{low_high}, \eqref{produt_reducing} and \eqref{reducing to B0} imply that for every $f\in \mathbb{Z}[\mathcal{GQ}_{2n-1}]$, there exist $F,G\in A[Q]$ such that 
				\[
				f=F+G\Delta_{B_{0}}+\mathfrak{I}.
				\]
				
				We now reduce the powers of $Q$ multiplying $\Delta_{B_{0}}$.
				Let $b_{r}:=n+r+1(=\overline{n-r})$, so that $B_{r+1}=B_{r}\setminus \{b_{r}\}$.
				Relation (iv) gives 
				\[
				M_{b_{r}}\Delta_{B_{r}}=\Delta_{B_{r+1}}.
				\]
				Therefore, we obtain
				\begin{align}\label{QBr}
					Q\Delta_{B_{r}}=\Delta_{B_{r+1}}+(Q-M_{b_{r}})\Delta_{B_{r}},
				\end{align}
				for $0\le r\le n-2$.
				Note that by Relation (ii), $Q-M_{b_{r}}=M_{n-r}-Q\in A$.
				For $B_{n-1}=\{2n\}$, Relation (i) gives 
				$M_{2n}\Delta_{B_{n-1}}=0.$
				Hence, by Relation (ii), we have 
				\begin{align}\label{QBn-1}
					Q\Delta_{B_{n-1}}=(Q-M_{2n})\Delta_{B_{n-1}}=(M_{1}-Q)\Delta_{B_{n-1}}.
				\end{align}
				Therefore, \eqref{QBr} and \eqref{QBn-1} imply repeatedly that, $G\in A[Q]\Delta_{B_{0}}$ can be written as  
				\begin{align}\label{G-part}
					G\Delta_{B_{0}}+\mathfrak{I}=\sum_{r=0}^{n-1}h'_{n+r+1}\Delta_{B_{r}}+\mathfrak{I},
				\end{align} 
				for some $h'_{n+r+1}\in A\simeq H^{*}(BT^{n})$ ($r=0,\ldots, n-1$).
				It remains to reduce $F\in A[Q]$. 
				Applying Relation (iii) to $B_{0}$, we obtain
				\[
				2\Delta_{B_{0}}=M_{1}\cdots M_{n}
				\]
				Since $M_{i}=Q+(M_{i}-Q)$, the right-hand side is a monic polynomial of degree $n$ in $A[Q]$. 
				Thus
				\begin{align}\label{Qn}
					Q^{n}=2\Delta_{B_{0}}+ \sum_{j=0}^{n-1}a_{j}Q_{j}
				\end{align}
				for some $a_{j}\in A$.
				Using \eqref{G-part}, the equation \eqref{Qn} reduces every power $Q^{m}$ ($m\ge n$), to an $A$-linear combination of
				\[
				1, Q,\ldots, Q^{n-1}, \Delta_{B_{0}}, \ldots, \Delta_{B_{n-1}}.
				\]
				Consequently, 
				\begin{align}\label{reducingtoQ}
					f=F+G\Delta_{B_{0}}+\mathfrak{I}=\sum_{j=0}^{n-1}h''_{j}Q^{j}+\sum_{r=0}^{n-1}h''_{n+r+1}\Delta_{B_{r}}+\mathfrak{I}.
				\end{align}
				for some $h_{j}''\in A$ ($j=0,\ldots, 2n$).
				Finally,
				\[
				M_{1}\cdots M_{j}=\prod_{i=1}^{j}(Q+(M_{i}-Q))=Q^{j}+\sum_{k=0}^{j-1}c_{jk}Q^{k}
				\]
				for some $c_{jk}\in A$.	
				Hence, we can change the generators from
				$1,Q, Q^{2}\ldots, Q^{n-1}$ to 
				\[
				1,M_{1}, M_{1}M_{2}\ldots, M_{1}\cdots M_{n-1}
				\]
				Therefore, \eqref{reducingtoQ} is equivalent to \eqref{key-equiation}. 
				This completes the proof. 
			\end{proof}

			We will establish some facts to prove the injectivity of $\varphi$ in Lemma~\ref{injectivitymainth}.
			
			Let $v\in \cv_{2n-1}=[2n]$. 
			For $n\ge 2$, we define $I_v\subset [n]\subset\cv_{2n-1}$ by 
			\[I_v=\begin{cases}
				[n]\setminus\{v\} &\text{ if } 1\leq v\leq n\\
				[n]\setminus\{\overline{v}\} &\text{ if } n+1\leq v\leq 2n.\\
			\end{cases}\]
			
			\begin{lema}\label{lemma5.4}
				Let $v\in\cv_{2n-1}$. 
				For $n\ge 2$, let 
				$\langle G_J~|~v\notin J\rangle$ be an ideal in $\bz[\mathcal{GQ}_{2n-1}]$ generated by $G_J$ (see \eqref{G_J}) with $v\notin J$. 
				Then, we have the following isomorphism:
				\begin{equation}\label{threeisom}
					\bz[\mathcal{GQ}_{2n-1}]/\langle G_J~|~v\notin J\rangle
					\simeq \bz[M_i, Q~|~i\in I_v]
					\simeq H^*(BT^n).
				\end{equation}
				For $n=1$, the following isomoprhism holds:
				\begin{equation}\label{threeisom_n=1}
					\bz[\mathcal{GQ}_{1}]/\langle M_{v}, \Delta_{\overline{v}}\rangle
					\simeq \bz[\Delta_{v}]
					\simeq H^*(BT^1),
				\end{equation}
				where $v\in \mathcal{V}_{1}=\{1,\overline{1}\}.$
			\end{lema}
			\begin{proof}
				For $n=1$, by using Relations 1--4, we have that
				\begin{equation*}
					\mathbb{Z}[\mathcal{GQ}_{1}]/\langle M_{v}, \Delta_{\overline{v}}\rangle\simeq \mathbb{Z}[M_{\overline{v}},\Delta_{v},Q]/\langle M_{\overline{v}}-2Q, 2\Delta_{v}-M_{\overline{v}},  \rangle\simeq \mathbb{Z}[\Delta_{v}].
				\end{equation*}
				This establishes the statement for $n=1$.
				
				We next assume that $n\ge 2$. We prove the statement only for the vertex $v=1\in \cv_{2n-1}$ because the proof for other vertices will follow similarly.
				If $v=1\in \cv_{2n-1}$, the isomorphism \eqref{threeisom} to be proved is stated as follows: 
				\begin{equation}
					\label{v=1}
					\bz[\mathcal{GQ}_{2n-1}]/\langle G_J~|~1\notin J\rangle\simeq \bz[M_2,\ldots,M_{n},Q]\simeq H^*(BT^n).
				\end{equation}
				
				Note that the following elements generate $\bz[\mathcal{GQ}_{2n-1}]/\langle G_J~|~1\notin J\rangle$:
				\begin{equation}\label{injproofgen1}
					\{\overline{Q}\}\cup\{\overline{M_v}: v\in\cv_{2n-1}\}\cup\{\overline{\Delta_P} : P\subset \cv_{2n-1} ,~ \{i,\overline{i}\}\not\subset P \text{ for all $i\in\cv_{2n-1}$}\}=:\{\overline{Q}\}\cup \overline{\mathbf{M}}\cup\overline{\mathbf{D}},\end{equation} where
				\begin{align*}
					& \overline{M_v}:=M_v+(\mathfrak{I}+\langle G_J~|~1\notin J\rangle), \\
					& \overline{Q}:=Q+(\mathfrak{I}+\langle G_J~|~1\notin J\rangle),\\
					& \overline{\Delta_P}:=\Delta_P+(\mathfrak{I}+\langle G_J~|~1\notin J\rangle).
				\end{align*}
				
				Let $L\subset \cv_{2n-1}$ such that $\{i,\overline{i}\}\not\subset L$ for every $1\leq i\leq n$. 
				If $1\notin L$, then $\Delta_L=G_L\in \langle G_J~|~1\notin J\rangle$. Therefore, $\overline{\Delta_L}=0$ in $\bz[\mathcal{GQ}_{2n-1}]/\langle G_J~|~1\notin J\rangle$.
				On the other hand, if $1\in L$ and $|L|=l<n$, then there exist vertices $v_{i}(\not=1)\in \mathcal{V}_{2n-1}$ $(i=1,\ldots, n-l)$ such that $v_{i}\not\in L$ and $v_{i}+v_{j}\not=2n+1$ for $1\le i<j\le n-l$. By repeatedly using Lemma \ref{lem-relation4} (Relation 4), we have the following equality:
				\begin{align*}
					\Delta_L&=\Delta_{L\cup \{v_{1}\}}\cdot M_{v_{1}} \\
					&=\Delta_{L\cup \{v_{1},v_{2}\}}\cdot M_{v_{1}}\cdot M_{v_{2}} \\
					&\quad \vdots \\
					&=\Delta_{L\cup \{v_{1},v_{2},\ldots, v_{n-l}\}}\cdot M_{v_{1}}\cdots M_{v_{n-l}}.
				\end{align*}
				This shows that every generators presented by $\overline{\Delta_P}\in \bz[\mathcal{GQ}_{2n-1}]/\langle G_J~|~1\notin J\rangle$  can be written in terms of some $\overline{M_{v}}$'s and $\overline{\Delta_L}$'s with $1\in L$ and $|L|=n$.
				Furthermore, for such an $L$, we can put $L=\{1,i_{1},\ldots, i_{n-1}\}$.  	Then by Lemma \ref{lemma-relation5errata}, we have that
					\[	\Delta_{\{1,i_{1},\ldots, i_{n-1}\}}
					+
					\Delta_{\{2n,i_{1},\ldots, i_{n-1}\}}
					=
					Q\prod_{i\in L^c\setminus\{2n\}}M_i.\]
					where we use $K=\cv_{2n-1}\setminus \{1, 2n, i_1,\ldots,i_{n-1}\}$, $a=1, ~\overline{a}=2n$.

%
%
				
			Since $1\not\in \{2n,i_{1},\ldots, i_{n-1}\}$, we have
				\begin{align*}
					\overline{\Delta_{L}}=\overline{Q}\prod_{i\in L^{c}\setminus\{2n\}}\overline{M_{i}}.
				\end{align*}
				This concludes that every element in $\overline{{\bf D}}$ can be expressed in terms of $\overline{Q}$ and the elements in $\overline{{\bf M}}$. 
				Namely, we can reduce the generators of \eqref{injproofgen1} to $\overline{Q}$ and $\overline{{\bf M}}$.
				
				Next, from \eqref{G_J}, we have $M_1=G_{\cv_{2n-1}\setminus\{1\}}$, i.e., $\overline{M_{1}}=0\in\bz[\mathcal{GQ}_{2n-1}]/\langle G_J~|~1\notin J\rangle$.
				Therefore, together with 
				Lemma~\ref{relation2} {(Relation 2)}, we have 
				\[2\overline{Q}=\overline{ M_{\overline{1}}}=\overline{M_2}+\overline{ M_{\overline{2}}}=\cdots=\overline{M_{n}}+ \overline{M_{\overline{n}}}
				\in \bz[\mathcal{GQ}_{2n-1}]/\langle G_J~|~1\notin J\rangle.\]
				Consequently,  \[2\overline{Q}=\overline{ M_{\overline{1}}},~~~
				\overline{M_{\overline{k}}}=2\overline{Q}-\overline{M_{k}}~~ \text{ for }k=2,\ldots,n.\]
				This implies that all generators in \eqref{injproofgen1} can be reduced to the following:
				\[\overline{Q} ~~\text{and}~~\{\overline{M_i}~|~2\leq i\leq n\}.\]
				Hence, the natural projection $p: \bz[Q, M_2, \ldots,M_{n}]\to 	\bz[\mathcal{GQ}_{2n-1}]/\langle G_J~|~1\notin J\rangle$ {(see \eqref{commdiag})} is surjective.
				Note that \[\bz[Q, M_2, \ldots,M_{n}]=\bz[-Q, M_2-Q, \ldots,M_{n}-Q].\]
				
				We finally consider the following composition homomorphism 
				\[\bz[-Q, M_2-Q, \ldots, M_n-Q]\xrightarrow{p} \frac{\bz[\mathcal{GQ}_{2n-1}]}{\langle G_J |~1\notin J\rangle }\xrightarrow{i_1}H^*(BT^n),\]
				where $i_1$ is induced from the composition $\bz[\mathcal{GQ}_{2n-1}]\xrightarrow{\varphi}H^*(\mathcal{GQ}_{2n-1})\to H^*(BT^{n})$, and is defined by $f\mapsto \varphi(f)(1)$ for $f\in\bz[\mathcal{GQ}_{2n-1}]$. 
				Namely, $i_{1}$ is the restriction map onto the vertex $1\in \mathcal{V}_{2n-1}$.
				Using \eqref{mapmv} and \eqref{gen_G}, we have 
				\[i_{1}\circ p(-Q)=\alpha_1, ~\text{ and, }~i_{1}\circ p(M_k-Q)=\alpha_k ~~\text{ for } k=2,\ldots,n.\]
				Hence, $i_1\circ p$ is an isomorphism, which shows that $p$ is injective. Consequently, $p$ is an isomorphism.
				
				Hence,	
				\begin{equation}
					\frac{\bz[\mathcal{GQ}_{2n-1}]}{\langle G_J |~1\notin J\rangle }\cong\bz[Q, M_2,\ldots,M_n]\cong H^*(BT^n)=\bz[\alpha_1,\ldots,\alpha_n].
				\end{equation}
				This establishes \eqref{v=1}.
			\end{proof}
			
			\begin{coro}
				\label{cor_H-localizataion}
				We have the following injective homomorphisms:
				\begin{eqnarray}
					\label{localization}
					& &\displaystyle H^*(\mathcal{GQ}_{2n-1})\hookrightarrow\bigoplus_{v\in\cv_{2n-1}}H^*(BT^{n})\simeq 
					\bigoplus_{v\in\cv_{2n-1}}\bz[Q, M_i~|~i\in I_v]\ {\text{for $n\ge 2$};} \\
					\label{localiation_n=1}
					& &\displaystyle {H^*(\mathcal{GQ}_{1})\hookrightarrow\bigoplus_{v\in\cv_{1}}H^*(BT^{1})\simeq	\mathbb{Z}[\Delta_{1}]\oplus \mathbb{Z}[\Delta_{\overline{1}}].}
				\end{eqnarray}
			\end{coro}
			\begin{proof}
				The first inclusion directly follows from the definition of graph equivariant cohomology of GKM graphs, while the the second part follows from Lemma \ref{lemma5.4}.
			\end{proof}
			
			\begin{lema}\label{injectivitymainth}
				The homomorphism $\phi: \bz[\mathcal{GQ}_{2n-1}]\to H^*(\mathcal{GQ}_{2n-1})$ is injective.
			\end{lema}
			\begin{proof}
				We first assume that $n=1$. Then, it follows from Relations 1--4 that
				\begin{align*}
					\mathbb{Z}[\mathcal{GQ}_{1}]\simeq \mathbb{Z}[\Delta_{1}, \Delta_{\overline{1}}]/\langle \Delta_{1}\Delta_{\overline{1}}\rangle.
				\end{align*}
				Therefore, every element in $\mathbb{Z}[\mathcal{GQ}_{1}]$ can be written as
				\begin{align}
					\label{splitting}
					k+f(\Delta_{1})+g(\Delta_{\overline{1}})
				\end{align}
				for some constant term $k\in\mathbb{Z}$ and two polynomials $f(\Delta_{1})\in \mathbb{Z}[\Delta_{1}]$ and $g(\Delta_{\overline{1}})\in \mathbb{Z}[\Delta_{\overline{1}}]$, where $f(\Delta_{1})$ and $g(\Delta_{\overline{1}})$ do not have constant terms.
				Therefore, together with Corollary~\ref{cor_H-localizataion}, we have that the following homomorphism is injective:
				\[\Phi:\displaystyle  \bz[\mathcal{GQ}_{1}]\xrightarrow{\phi}H^*(\mathcal{GQ}_{1})\hookrightarrow\bigoplus_{v\in\cv_{1}}H^*(BT^{1})\simeq\bz[\Delta_{1}]\oplus \bz[\Delta_{\overline{1}}],\] 
				where $k$ in \eqref{splitting} maps to the diagonal element. This establishes that $\varphi$ is injective.
				
				We next assume that $n\ge 2$. Similarly, we claim that the following map $\Phi$ is injective:
				\[\Phi:\displaystyle  \bz[\mathcal{GQ}_{2n-1}]\xrightarrow{\varphi}H^*(\mathcal{GQ}_{2n-1})\hookrightarrow\bigoplus_{v\in\cv_{2n-1}}H^*(BT^{n})\simeq\bigoplus_{v\in\cv_{2n-1}}\bz[Q, M_i~|~i\in I_v].\] 
				
				Let \[r_u:\displaystyle \bigoplus_{v\in\cv_{2n-1}}\bz[Q, M_i~|~i\in I_v]\to\bz[Q, M_i~|~i\in I_u]\] be the restriction map at $u\in\cv_{2n-1}$. 
				For any $f\in \bz[\mathcal{GQ}_{2n-1}]$, let $f(u)$ denote the image of  $f$ by $r_u\circ \Phi$. Assume that $\Phi(f)=0$ for an element $f\in \bz[\mathcal{GQ}_{2n-1}]$. Therefore, 
				\begin{equation}\label{def-f}
					r_v\circ \Phi(f)=f(v)=0 \in \bz[Q, M_i~|~i\in I_v]~~~\text{for all } v\in\cv_{2n-1}.\end{equation}
				Here, the element $f\in \bz[\mathcal{GQ}_{2n-1}]$ can be written as follows (see  \eqref{formatcohomclass} and Lemma \ref{surjremark1}):
				\begin{equation}\label{inj1}
					\begin{split}
						f=h_1+h_2M_1&+\cdots+h_{n}M_1\cdots M_{n-1}\\&+h_{n+1}\Delta_{\{n+1,\ldots,2n\}}+\cdots+h_{2n}\Delta_{\{2n\}}+\mathfrak{I},
					\end{split}
				\end{equation}
				where $h_i\in H^*(BT^n)\subset\bz[{\bf M}, ~{\bf D},~ Q]$ for all $i=1,\ldots, 2n$.
				Note that, for all $i=1,\ldots,2n$, $\varphi(h_i)\in \bz[\alpha_1,\ldots,\alpha_n]$ (see \eqref{2.2}). Consequently, if there exists a vertex $u\in\mathcal{V}_{2n-1}$ such that $h_i(u)=0$, then $h_i$ must be identically zero, as $\varphi(h_i)$ is a constant in $H^*(\mathcal{GQ}_{2n-1})$.
				
				Since 
				\begin{equation}\label{inj2}
					\Delta_{\{n+1,\ldots,2n\}}(i)=\cdots=\Delta_{\{2n\}}(i)=0 ~~\text{for all}~~ i=1,\ldots,n,
				\end{equation}
				and $M_1(1)=0$,  
				using \eqref{def-f} and \eqref{inj1}, we have
				\[0=f(1)=h_1(1)+ 0+\cdots+0.\]
				Hence, $h_1=0$ since $h_1$ is a constant map (cf. \eqref{eqn3.10}).
				Furthermore, by plugging $h_1=0$ into \eqref{inj1}, and using $M_2(2)=0$, we have 
				\[0=f(2)=h_2M_1(2)+ 0+\cdots+0 .\] 
				Notice that  ${h_{2}M_{1}(2)=h_2(2) M_1(2)}\in \bz[Q,~M_i~|~i\in I_2]$ (see \eqref{localization}) and $M_1(2)\not=0$. Hence $h_2(2)=0$ since $\bz[{Q}, M_i~|~i\in I_2]$ is an integral domain. Therefore, $h_2=0$.
				By following similar arguments for $i=3,\ldots,n$, one can have \[h_3=\cdots=h_{n}=0.\]
				Therefore, by \eqref{inj1},
				\begin{equation}\label{inj3}
					\begin{split}
						f=h_{n+1}\Delta_{\{n+1,\ldots,2n\}}+\cdots+h_{2n-1}\Delta_{\{2n-1,2n\}}+h_{2n}\Delta_{\{2n\}}+\mathfrak{I}.
					\end{split}
				\end{equation}
				From \eqref{inj3}, we next obtain the following equality for the vertex $n+1\in\cv_{2n-1}$:
				\[	0=f(n+1)=h_{n+1}\Delta_{\{n+1,\ldots,2n\}}(n+1)+h_{n+2}\Delta_{\{n+2,\ldots,2n\}}(n+1)+\cdots+h_{2n}\Delta_{\{2n\}}(n+1).\]
				Since
				\[	\begin{split}
					\Delta_{\{n+2,\ldots,2n\}}&(n+1)=\cdots=\Delta_{\{2n\}}(n+1)=0,\\
					&\Delta_{\{n+1,\ldots,2n\}}(n+1)\not=0,    
				\end{split}\]
				we have $h_{n+1}=0$, by the same reason as above. 
				Iterating similar arguments for the other vertices $n+2,\ldots,2n\in\cv_{2n-1}$, we have  $h_i=0$ for all $n+2\leq i\leq 2n$. 
				
				Therefore $f=0$ in $\bz[\mathcal{GQ}_{2n-1}]=\bz[{\bf M,~D}, ~Q]/\mathfrak{I}$, which shows the injectivity of $\Phi$.	Hence, the lemma follows. 
			\end{proof}
			
			{\it Proof of Theorem \ref{mainth}. }
			The result \eqref{mainth1} follows from Lemmas \ref{surjmainth} and \ref{injectivitymainth}. Moreover, by Lemma \ref{Lemma3.1}, we obtain the isomorphism \eqref{mainth2}. \qed
			
			
			\begin{rema}
				In \cite{ghz06}, Guillemin, Holm and Zara study the equivariant cohomology ring of a homogeneous
				space using a combinatorial approach. 
				They obtain the Betti numbers for a broad class of graph equivaraint cohomologies associated with homogeneous spaces, and establish an isomorphism between the Borel description and the graph equivariant cohomology.
				In contrast, our work provides explicit generators and relations for the graph equivariant cohomology associated with the odd-dimensional complex quadric. 
				In addition, we identify the generators corresponding to the subgraphs of the GKM graph.
			\end{rema}
			
			\section{Ordinary cohomology of odd dimensional complex quadrics}
			\label{sect:6}
			
			In the paper \cite{ku23}, we compare the ordinary cohomologies of $H^{*}(Q_{4n})$ and $H^{*}(Q_{4n-2})$ by using the graph equivariant cohomology. 
			In this section, we compute the ordinary cohomology $H^{*}(Q_{2n-1})$ from the graph equivariant cohomology $H^{*}(\mathcal{GQ}_{2n-1})$.
			
			Recall the ordinary cohomology ring formulas of $Q_{m}$.
			According to \cite[Appendix C.3.4]{eh13}, we have  
			\begin{align*}
				& H^{*}(Q_{2n})\simeq \mathbb{Z}[c,x]/\langle c^{n+1}-2cx, x^{2}- \delta(n)c^{n}x \rangle, \\
				& H^{*}(Q_{2n-1})\simeq \mathbb{Z}[c,x]/\langle c^{n}-2x, x^{2} \rangle,
			\end{align*}
			where $\deg c=2$, $\deg x=2n$ and 
			\begin{align*}
				\delta(n)=
				\left\{
				\begin{array}{ll}
					0 & n\equiv 1 \mod 2 \\
					1 & n\equiv 0 \mod 2 
				\end{array}
				\right.
			\end{align*}
			
			Since $H^{odd}(Q_{2n-1})=0$, the odd-dimensional complex quadric $Q_{2n-1}$ is the equivariantly formal GKM manifold (see \cite{gkm}).
			Therefore, its ordinary cohomology is isomorphic to the quotient of $H_{T}^{*}(Q_{2n-1})$ by the ideal generated by $\pi^{*}(\alpha_{1}),\ldots, \pi^{*}(\alpha_{n})$, where $\alpha_{1},\ldots, \alpha_{n}\in H^{2}(BT^{n})$ are generators and $\pi^{*}:H^{*}(BT^{n})\to H_{T}^{*}(Q_{2n-1})$ is the induced (injective) homomorphism from the projection $\pi:ET^{n}\times_{T^{n}}Q_{2n-1}\to BT^{n}$.
			Recall that the equivariant cohomology $H_{T}^{*}(Q_{2n-1})$ is defined by the ordinary cohomology of the Borel construction $ET^{n}\times_{T^{n}}Q_{2n-1}$.
			Thus, using Theorem~\ref{mainth} and Proposition~\ref{H(BT)-generators}, we also have the ordinary cohomology of $Q_{2n-1}$ as follows.
			\begin{coro}
				\label{main-corollary}
				The ordinary cohomology $H^{*}(Q_{2n-1})$ is isomorphic to $\mathbb{Z}[\mathcal{GQ}_{2n-1}]/\mathcal{J}$, where $\mathcal{J}$ is generated by
				\begin{align*}
					M_{i}-Q
				\end{align*}
				for $i=1,\ldots, n$.
			\end{coro}
			
			To reduce the relations of $\mathbb{Z}[\mathbf{M}, \mathbf{D}, Q]$, 
			we first show the following lemma.
			\begin{lema}
				\label{final-lemma}
				If $K, H\subset \mathcal{V}_{2n-1}$ are the subsets such that  
				$|K|=|H|=n$ and $\{i, \overline{i}\}\not\not\subset K, H$ for every $i\in \mathcal{V}_{2n-1}$,
				then, there is the following formula in $\mathbb{Z}[\mathcal{GQ}_{2n-1}]/\mathcal{J}$:
				\begin{equation}\label{Lem:6.2_1}
										\Delta_{K}=\Delta_{H}.
				\end{equation}
				
				Furthermore, if $L\subset \mathcal{V}_{2n-1}$ satisfies 
				$|L|=n-l$ for some $1\le l\le n$ and $\{i, \overline{i}\}\not\not\subset L$ for every $i\in \mathcal{V}_{2n-1}$, then 
				\begin{align*}
					\Delta_{L}=\Delta_{K}Q^{l}.
				\end{align*}
			\end{lema} 
			\begin{proof}
				By the definition of $\mathcal{J}$ and Lemma~\ref{relation2} (Relation~2), 
				in $\mathbb{Z}[\mathcal{GQ}_{2n-1}]/\mathcal{J}$, we have 
				\begin{align}
					\label{deg2-ordinary-generator}
					Q=M_{1}=M_{2}=\cdots =M_{n}=M_{n+1}=\cdots =M_{2n}.
				\end{align}
			Moreover, by using Lemma~\ref{lemma-relation3} (Relation 3) and  Lemma~\ref{lemma-relation5errata} (Relation 5), we have that
					\begin{align*}
					2\Delta_{K}=\prod_{i\in K^c}M_i	&=Q^n=Q\prod_{i\in (K\cup\{\overline{a}\})^c}M_i=	\Delta_{K}	+ \Delta_{(K\setminus\{a\})\cup\{\overline{a}\}}.\\
					\implies &\Delta_{K}	= \Delta_{(K\setminus\{a\})\cup\{\overline{a}\}}.
				\end{align*}
				for any $a\in K$. From the assumptions on $K$ and $H$, it is readily observed that $H$ can be obtained from $K$ by finitely many replacements of an element $a$ by $\overline{a}$. Hence, the first statement follows by iterating the above identity.
				
			The 2nd statement follows by using 
				\eqref{Lem:6.2_1}, \eqref{deg2-ordinary-generator}
				and Relation 4 repeatedly. 
\end{proof}
			
			By Lemma~\ref{surjremark1}, Lemma~\ref{final-lemma} and \eqref{deg2-ordinary-generator}, every element $f\in \mathbb{Z}[\mathcal{GQ}_{2n-1}]/\mathcal{J}$ can be written as 
			\begin{align*}
				f=k_{0}+k_{1}Q+\cdots + k_{n-1}Q^{n-1}+k_{n}\Delta_{K}+k_{n+1}\Delta_{K}Q+\cdots +k_{2n-1}\Delta_{K}Q^{n-1},
			\end{align*}
			for some unique $k_{0},\ldots, k_{2n-1}\in \mathbb{Z}$.
			Therefore, there is the following isomorphism as a $\mathbb{Z}$-module:
			\begin{align*}
				&\mathbb{Z}[\mathcal{GQ}_{2n-1}]/\mathcal{J}\simeq \mathbb{Z}\oplus \mathbb{Z}Q+\cdots \oplus \mathbb{Z}Q^{n-1}\oplus \mathbb{Z}\Delta_{K}\oplus \mathbb{Z}\Delta_{K}Q\oplus \cdots \oplus \mathbb{Z}\Delta_{K}Q^{n-1} \\
				& \simeq H^{*}(Q_{2n-1})\simeq \mathbb{Z}\oplus \mathbb{Z}c+\cdots \oplus \mathbb{Z}c^{n-1}\oplus \mathbb{Z}x\oplus \mathbb{Z}xc\oplus \cdots \oplus \mathbb{Z}xc^{n-1}.
			\end{align*} 
			By Lemma~\ref{lemma-relation3} (Relation 3), we have that 
			\begin{align*}
				Q^{n}=\prod_{i\in K^{c}}M_i^n=2\Delta_{K}.
			\end{align*}
	Further, since $K\cap K^c=\emptyset$,  we  have the following from Lemma~\ref{lem-relation1} (Relation 1) and \eqref{Lem:6.2_1}
				\begin{align*}
		0=\Delta_{K}\Delta_{K^{c}}=\Delta_{K}^{2}.
			\end{align*}
			Consequently we have the ordinary cohomology.
			\begin{propo}
				\label{final-prop}
				There is the following isomorphism :
				\begin{align*}
					H^{*}(Q_{2n-1})\simeq \mathbb{Z}[Q,\Delta_{K}]/\langle Q^{n}-2\Delta_{K}, \Delta_{K}^{2} \rangle,
				\end{align*}
				where $\deg Q=2$ and $\deg \Delta_{K}=2n$.
			\end{propo}
			
			\section{GKM graphs of even- and odd-dimensional complex quadrics}
			\label{sect:7}
			
			By Remark~\ref{rem_submanifold_evem-dim-quadric}, there is a GKM subgraph $\mathcal{GQ}_{2n-2}$ of $\mathcal{GQ}_{2n-1}$. 
			This induces a homomorphism (see Section~\ref{sect:7.2} for details)
			\begin{align*}
				H^{*}(\mathcal{GQ}_{2n-1})\to H^{*}(\mathcal{GQ}_{2n-2}).
			\end{align*}
			In this section, we compare the graph equivariant cohomologies induced by even- and odd-dimensional complex quadrics.
			
			We first recall the restricted $T^n$-action of \eqref{introtorusaction} on the even-dimensional complex quadric
			\begin{align*}
				Q_{2n-2}:=\Big\{[z_1:\cdots:z_{2n}:0]\in\mathbb{CP}^{2n}~|~\sum_{i=1}^{n}z_iz_{2n+1-i}=0\Big\}\subset Q_{2n-1}.
			\end{align*} 
			By restricting \eqref{introtorusaction} to $Q_{2n-2}$ we have the following action 
			\begin{equation}\label{sect7torusaction}
				[z_1:\cdots:z_{2n}:0]\cdot(t_1,\ldots,t_{n}):=[z_1t_1:\cdots:z_{n}t_{n}:t_{n}^{-1}z_{n+1}:\cdots:t_1^{-1}z_{2n}:0].
			\end{equation}
			We use the notation $(Q_{2n-2}, T^{n})$ to denote the $T^n$-action on $Q_{2n-2}$ defined by \eqref{sect7torusaction}.
			It is easy to check that the kernel of $(Q_{2n-2}, T^{n})$, i.e., the intersection of isotropy subgroups of all elements in $Q_{2n-2}$, is
			$\Delta(\mathbb{Z}_{2}):=\{(1,\ldots,1), (-1,\ldots,-1)\}\subset T^{n}$. 
			Therefore, this action is not effective.
			
			\subsection{The GKM graph $\mathcal{GQ}_{2n-2}$ induced from the non-effective $T^{n}$-action on $Q_{2n-2}$}
			\label{sect:7.1} 
			In this subsection, we compute the GKM graph of $(Q_{2n-2},T^{n})$ (cf. the GKM graph of $Q_{2n-2}$ with the effective torus $T^{n}$-actions in \cite[Section 2.1]{ku23}).
			Let $\Gamma_{2n-2}\subset \Gamma_{2n-1}$ be the subgraph consisting of: 
			\begin{description}
				\item[Vertices] $\mathcal{V}_{2n-2}:=\mathcal{V}_{2n-1}$;
				\item[Edges] $\mathcal{E}_{2n-2}:=\mathcal{E}_{2n-1}\setminus \{i\overline{i}\ |\ 1\le i\le n\}$.
			\end{description}
			The GKM graph $\mathcal{GQ}_{2n-2}:=(\Gamma_{2n-2}, \alpha')$ is defined by the restricted axial function
			\begin{align*}
				\alpha':=\alpha|_{\mathcal{E}_{2n-2}}:\mathcal{E}_{2n-2}\to H^{2}(BT^{n}),
			\end{align*}
			where $\alpha:\mathcal{E}_{2n-1}\to H^{2}(BT^{n})$ is the axial function of $\mathcal{GQ}_{2n-1}:=(\Gamma_{2n-1}, \alpha)$ (see Section~\ref{sect:2.1}).
			More precisely, it is defined by the following equations:
			\begin{itemize}
				\item $\alpha'(ij)=\alpha(ij)=-\alpha_i+\alpha_j$ for $1\leq i\neq j\leq n$;
				\item $\alpha'(i\overline{j})=\alpha(i\overline{j})=-\alpha_i-\alpha_j$ for $1\leq i\neq j\leq n$;
				\item $\alpha'(\overline{i}j)=\alpha(\overline{i}j)=\alpha_i+\alpha_j$ for $1\leq i\neq j\leq n$;
				\item $\alpha'(\overline{i}\overline{j})=\alpha(\overline{i}\overline{j})=\alpha_i-\alpha_j$ for $1\leq i\neq j\leq n$.
			\end{itemize}
			Since $(\Gamma_{2n-2},\alpha')$ is defined by the restriction of $(\Gamma_{2n-1},\alpha)$, there is an inclusion from $\mathcal{GQ}_{2n-2}$ to $\mathcal{GQ}_{2n-1}$ (cf. Figure \ref{lowest-example for inclusion} for an example)
			We denote it by 
			\begin{align*}
				\iota:\mathcal{GQ}_{2n-2}\hookrightarrow \mathcal{GQ}_{2n-1}.
			\end{align*}
			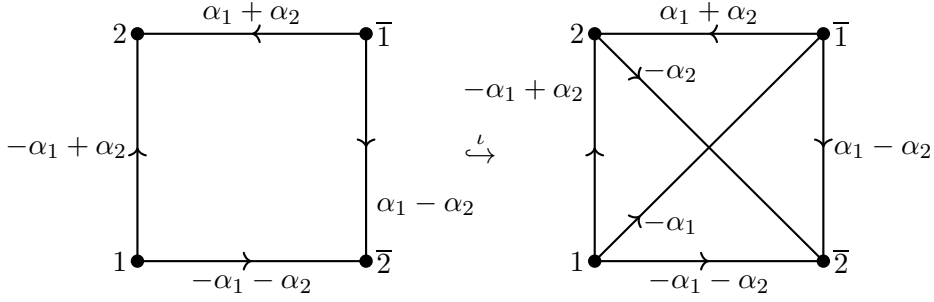
\begin{figure}[H]
				\begin{tikzpicture}
					\begin{scope}[xscale=0.5, yscale=0.5]
						\coordinate (a) at (-15,-3);
						\coordinate (b) at (-15,3);
						\coordinate (c) at (-9,3);
						\coordinate (d) at (-9,-3);
						
						\fill(a) circle (5pt);
						\node[left] at (a) {$1$};
						\fill(b) circle (5pt);
						\node[left] at (b) {$2$};
						\fill(c) circle (5pt);
						\node[right] at (c) {$\overline{1}$}; 
						\fill(d) circle (5pt);
						\node[right] at (d) {$\overline{2}$}; 
						
						\draw[thick, decoration={markings,
							mark=at position 0.5 with \arrow{>}},
						postaction=decorate](a)--(b);
						\node[left] at (-15,0) {$-\alpha_{1}+\alpha_{2}$};

						\draw[thick, decoration={markings,
							mark=at position 0.5 with \arrow{>}},
						postaction=decorate](a)--(d);
						\node[right] at (-9,-1.5) {$\alpha_{1}-\alpha_{2}$};
						
						\draw[thick, decoration={markings,
							mark=at position 0.5 with \arrow{>}},
						postaction=decorate](c)--(b);
						\node[above] at (-12,3) {$\alpha_{1}+\alpha_{2}$};

						\draw[thick, decoration={markings,
							mark=at position 0.5 with \arrow{>}},
						postaction=decorate](c)--(d);
						\node[below] at (-12,-3) {$-\alpha_{1}-\alpha_{2}$};

						\node at (-6,0) {$\stackrel{\iota}{\hookrightarrow}$};

						\coordinate (1) at (-3,-3);
						\coordinate (2) at (-3,3);
						\coordinate (3) at (3,3);
						\coordinate (4) at (3,-3);
						
						\fill(1) circle (5pt);
						\node[left] at (1) {$1$};
						\fill(2) circle (5pt);
						\node[left] at (2) {$2$};
						\fill(3) circle (5pt);
						\node[right] at (3) {$\overline{1}$}; 
						\fill(4) circle (5pt);
						\node[right] at (4) {$\overline{2}$}; 
						
						\draw[thick, decoration={markings,
							mark=at position 0.5 with \arrow{>}},
						postaction=decorate](1)--(2);
						\node[left] at (-3,1.5) {$-\alpha_{1}+\alpha_{2}$};
						
						\draw[thick, decoration={markings,
							mark=at position 0.2 with \arrow{>}},
						postaction=decorate](1)--(3);
						\node[right] at (-2,-2) {$-\alpha_{1}$};
						
						\draw[thick, decoration={markings,
							mark=at position 0.5 with \arrow{>}},
						postaction=decorate](1)--(4);
						\node[right] at (3,0) {$\alpha_{1}-\alpha_{2}$};
						
						\draw[thick, decoration={markings,
							mark=at position 0.5 with \arrow{>}},
						postaction=decorate](3)--(2);
						\node[above] at (0,3) {$\alpha_{1}+\alpha_{2}$};
						
						\draw[thick, decoration={markings,
							mark=at position 0.2 with \arrow{>}},
						postaction=decorate](2)--(4);
						\node[right] at (-2,2) {$-\alpha_{2}$};
						
						\draw[thick, decoration={markings,
							mark=at position 0.5 with \arrow{>}},
						postaction=decorate](3)--(4);
						\node[below] at (0,-3) {$-\alpha_{1}-\alpha_{2}$};
						
					\end{scope}
				\end{tikzpicture}
				\caption{The inclusion of GKM graphs $\iota:\mathcal{GQ}_{2}\hookrightarrow \mathcal{GQ}_{3}$ for $n=2$, where $\alpha_{1}$ and $\alpha_{2}$ are generators in $H^{2}(BT^{2})$.}
				\label{lowest-example for inclusion}
			\end{figure}

			\subsection{The equivariant cohomology of $H^{*}_{T}(Q_{2n-2})$}
			\label{sect:7.2} 
			In \cite{ku23}, the equivariant cohomology of $Q_{2n-2}$ with effective $T^{n}$-actions is computed using the graph equivariant cohomology of its GKM graph. 
			In this subsection, in Lemma~\ref{key-lemma}, we prove that the equivariant cohomology of the non-effective $(Q_{2n-2},T^{n})$ is also computed using the graph equivariant cohomology of $\mathcal{GQ}_{2n-2}$.
			
			The following fact from \cite[Corollary 2.2]{fp11} is the essential result to prove it.
			\begin{guess}[Franz-Puppe]
				\label{FPT}
				Let $M$ be a $T$-space such that $H^{*}_{T}(X)$ is free over $H^{*}(BT)$.
				If the isotropy group of each $x\not\in X_{1}:=\{x\in X\ |\ \dim T(x)\le 1\}$ is contained in a proper subtorus of $T$, then the following sequence is exact:
				\begin{align*}
					0\longrightarrow H_{T}^{*}(X;\mathbb{Z})\stackrel{i^{*}}{\longrightarrow} H_{T}^{*}(X^{T};\mathbb{Z})\stackrel{\delta}{\longrightarrow} H_{T}^{*+1}(X_{1},X^{T};\mathbb{Z}),
				\end{align*}
				where the homomorphism $i^{*}$ is induced by the inclusion $i:X^T\to X$, and the homomorphism $\delta$ is the connecting homomorphism of the long exact sequence of cohomology of the pair $(X_{1},X^{T})$.
			\end{guess}
			By using this theorem, we obtain the following lemma.
			\begin{lema}\label{FPT2}
				\label{key-lemma}
				The equivariant cohomology $H_{T^n}^{*}(Q_{2n-2})$
				is isomorphic to 
				the graph equivariant cohomology $H^{*}(\mathcal{GQ}_{2n-2})$, i.e., 
				\begin{align}
					H_{T^{n}}^{*}(Q_{2n-2})&\simeq 
					H^{*}(\mathcal{GQ}_{2n-2}) \nonumber \\
					&:=\{f:\mathcal{V}_{2n-2}\to H^{*}(BT)\ |\ f(p)-f(q)\equiv 0 \mod\alpha(pq),\  \text{for all $pq\in \mathcal{E}_{2n-2}$} \} \label{congrel_v2}.
				\end{align}
			\end{lema}
			\begin{proof}
				Since $H^{odd}(Q_{2n-2})=0$ (see Section~\ref{sect:6}), it follows from the spectral sequence argument that $H^{*}_{T}(Q_{2n-2})$ is free over $H^{*}(BT)$.
				Moreover, it is easy to check that an element $x\in Q_{2n-2}$ with its orbit $T(x)\simeq T^{n-k}$ ($k\ge n-2$) has the isotropy group $T_{x}\simeq T^{k}\times \mathbb{Z}_{2}$.
				Therefore, there is the proper subtorus $K\subset T^{n}$ which is isomorphic to $T^{k+1}$ such that $T_{x}\subset K$.
				For example, the point $x=[z_{1}:\ldots :z_{n}:0:\ldots:0]$ with $z_{i}\not=0$ for $i=1,\ldots, n$ has the isotropy subgroup $\Delta(\mathbb{Z}_{2})$, and $\Delta(\mathbb{Z}_{2})$ is contained in the proper subtorus $K=\{(t,\ldots, t)\ |\ t\in S^{1}\}\subset T^{n}$, where $K\simeq T^{1}$.
				This shows that $(Q_{2n-2},T^{n})$ satisfies the assumptions of Theorem~\ref{FPT}.
				Therefore, we have 
				\begin{align*}
					H^{*}_{T^n}(Q_{2n-2})\simeq \im i^{*}=\ker \delta.
				\end{align*}
				Because $\delta$ is the connecting homomorphism of the long exact sequence of $(X_{1},X^{T})$ for $X=Q_{2n-2}$, we have 
				\begin{align*}
					H^{*}_{T^n}(Q_{2n-2})\simeq \ker\delta= \im j^{*},
				\end{align*}
				where $j:X^{T}\to X_{1}$ is the inclusion from the fixed points $X^{T}$ to $X_{1}$ and 
				$j^{*}:H_{T}^{*}(X_{1})\to H_{T}^{*}(X^{T})$ is its induced homomorphism.
				Using the method with Mayer-Vietoris sequence demonstrated in e.g.~\cite{bfr09, f10} (also see the proof of \cite[Theorem 2.9]{dks22}), we have that 
				\begin{align*}
					\im j^{*}\simeq \{f:\mathcal{V}_{2n-2}\to H^{*}(BT)\ |\ f(p)-f(q)\equiv 0 \mod\alpha(pq),\  \text{for all $pq\in \mathcal{E}_{2n-2}$} \}.
				\end{align*}  
				This completes the proof.
			\end{proof}
			
			\subsection{The graph equivariant cohomology of $\mathcal{GQ}_{2n-2}$ and the induced homomorphism $\iota^{*}$}
			\label{sect:7.3} 
			By \eqref{congrel} and \eqref{congrel_v2}, the inclusion $\iota:\mathcal{GQ}_{2n-2}\hookrightarrow \mathcal{GQ}_{2n-1}$ induces the following homomorphism:
			\begin{align*}
				\iota^{*}:H^*(\mathcal{GQ}_{2n-1})\to H^*(\mathcal{GQ}_{2n-2})
			\end{align*}
			such that 
			\begin{align*}
				\iota^{*}(f):=f',
			\end{align*}
			where $f'(p):=f(p)$ for all $p\in \mathcal{V}_{2n-2}=\mathcal{V}_{2n-1}$.
			In fact, since $f$ satisfies the congruence relations for every edge in $\mathcal{E}_{2n-1}$, 
			its restriction $f'$ also satisfies the congruence relations for every edge in $\mathcal{E}_{2n-2}$; therefore, $f'\in H^{*}(\mathcal{GQ}_{2n-2})$.
			We call $\iota^{*}$ an {\it induced homomorphism} from $\iota$.
			
			Note that the comcepts described above are also defined for any GKM graph $\mathcal{G}$ and its GKM subgraph $\mathcal{G}'$.
			That is, if there is an inclusion of a GKM subgraph $\iota:\mathcal{G}'\hookrightarrow \mathcal{G}$ then there is the induced homomorphism $\iota^{*}:H^{*}(\mathcal{G})\to H^{*}(\mathcal{G}')$.
			We have the following lemma:
			\begin{lema}
				Let $\mathcal{G}$ be a GKM graph. 
				Assume that there is a GKM subgraph $\mathcal{G}'$ of $\mathcal{G}$ such that the set of vertices of $\mathcal{G}'$ coincides with that of $\mathcal{G}$.
				Then the induced homomorphism $\iota^{*}:H^{*}(\mathcal{G})\to H^{*}(\mathcal{G}')$ is injective. 
			\end{lema}
			\begin{proof}
				Let $\mathcal{V}$ be the set of vertices of $\mathcal{G}$ and $\mathcal{G}'$.
				By the definition of graph equivariant cohomology (see e.g. \cite{gz01}), 
				the restriction to the vertices 
				\begin{align*}
					H^{*}(\mathcal{G})\stackrel{\iota^{*}}{\to} H^{*}(\mathcal{G}')\to \bigoplus_{v\in \mathcal{V}}H^{*}(BT)
				\end{align*} 
				is injective.
				Therefore, the homomorphism $\iota^{*}$ is injective.
			\end{proof}
			
			Since $\mathcal{V}_{2n-2}=\mathcal{V}_{2n-1}$ for $\iota:\mathcal{GQ}_{2n-2}\hookrightarrow \mathcal{GQ}_{2n-1}$, we have the following corollary:
			\begin{coro}
				\label{injectivity_equiv_cohom}
				The induced homomorphism $\iota^{*}:H^{*}(\mathcal{GQ}_{2n-1})\to H^{*}(\mathcal{GQ}_{2n-2})$ is injective. 
			\end{coro}
			
			\subsection{The ring structure of $H^{*}(\mathcal{GQ}_{2n-2})$ and the image ${\rm Im}~\iota^{*}$}
			\label{sect:7.4} 
			
			In this subsection, we compute $H^{*}(\mathcal{GQ}_{2n-2})$, and then determine the image of $\iota^{*}:H^{*}(\mathcal{GQ}_{2n-1})\to H^{*}(\mathcal{GQ}_{2n-2})$.
			
			For $H^{*}(\mathcal{GQ}_{2n-2})$, 
			the following elements will serve as generators. 
			Some of them can be defined by using elements in $H^{*}(\mathcal{GQ}_{2n-1})$ (i.e., the functions $\mathcal{V}_{2n-1}(=\mathcal{V}_{2n-2})\to H^{*}(BT^{n})$ that satisfies \eqref{congrel}).
			
			\begin{description}
				\item[Generator 1] $\mathcal{M}:=\{M'_{v}:\mathcal{V}_{2n-2}\to H^{2}(BT^{n})\ |\ v\in \mathcal{V}_{2n-2}\}$ where $M_v':=\iota^{*}(M_v)$. More precisely, by  \eqref{mapmv}, 
				\[M'_v(j)=\begin{cases}
					0 &\text{if $j=v$}\\
					\alpha(jv) =-\alpha_j+\alpha_v&\text{if $j\not=v, \overline{v}$}\\
					2\alpha_v&\text{if $j= \overline{v}$}
				\end{cases}
				\]
				\item[Generator 2] $\mathcal{D}:=\{\Delta'_{K}:\mathcal{V}_{2n-2}\to H^{4n-2|K|-2}(BT^{n})\ |\ K\subset \cv_{2n-2},\ \text{ $\{i,\overline{i}\}\not\subset K$ $\forall i\in \cv_{2n-2}$}\}$  where \begin{equation}\label{deltaleven}\Delta'_{K}(j)=\begin{cases}
						\displaystyle\prod_{k\notin K\cup\{\overline{j}\}}\alpha(jk) & \text{if $j\in K$},\\
						~~~~~~~~~~~~~0& \text{if $j\notin K$}.
				\end{cases}\end{equation}
				\item[Generator 3] $X:=\iota^{*}(Q):\mathcal{V}_{2n-2}\to H^{2}(BT^{n})$ such that $X(i):=Q(i)=-\alpha_{i}$ for $1\le i\le 2n$.
			\end{description}
			
			\begin{rema}
				\label{remark-effective-noneffective}
				The element $X$ also appears in the graph equivariant cohomology of the GKM graph in \cite[Section 4]{ku23}, which is induced from the effective $T^{n}$-action on $Q_{2n-2}$.
				In \cite{ku23}, $X$ is not needed as a generator because $M_{v}+M_{\overline{v}}=X$ holds (see Relation 2 of \cite{ku23}).
				However, in the GKM graph induced from the non-effective action, 
				for any $v\in\cv_{2n-2}$, the following equality holds:
				\begin{equation}\label{2X-relation}
					M'_v+M'_{\overline{v}}=2X.
				\end{equation}
				Therefore, $X$ is needed as a generator in the non-effective case because of \eqref{2X-relation}. 
			\end{rema}
			
			\begin{rema} 
				Let $K\subset\cv_{2n-2}$ be a subset such that  $\{i,\overline{i}\}\not\subset K$ for any $i\in \cv_{2n-2}$. 
				We note that $\Delta_{K}'\in H^{4n-2|K|-2}(\mathcal{GQ}_{2n-2})$ is not induced from $\Delta_{K}\in H^{4n-2|K|}(\mathcal{GQ}_{2n-1})$ (see Definition~\ref{Delta_k}), i.e., $\iota^{*}(\Delta_{K})\not=\Delta'_{K}$, because their degrees are different (see Figure~\ref{comparison_of_two_elements}).
			\end{rema}
			
			Let $\mathbb{Z}[\mathcal{M}, \mathcal{D}, X]$ be the polynomial ring generated by all elements in $\mathcal{M}$, $\mathcal{D}$ and $X$. 
			Let $\mathcal{I}$ be the ideal in $\mathbb{Z}[\mathcal{M}, \mathcal{D}, X]$ generated by the following four types of relations:
			\begin{description}
				\item[Relation 1] $\prod_{\cap J=\emptyset}G_{J}$ for $G_{J}\in \mathcal{M}\sqcup \mathcal{D}$, where $G_{J}$ is defined similarly to \eqref{G_J};
				\item[Relation 2] $M'_v+M'_{\overline{v}}-2X$  for any $v\in\cv_{2n-1}$; 
				\item[Relation 3] $\prod_{i\in I}M'_{i}-(\Delta'_{(I\cup \{a\})^{c}}+\Delta'_{(I\cup \{\overline{a}\})^{c}})$ for every subset $I\subset \mathcal{V}_{2n-2}$ such that $|I|=n-1$ and there exists a unique pair $\{a,\overline{a}\}\subset I^c$;
				\item[Relation 4] $\Delta'_{K}\cdot M'_{i}-\Delta'_{K\setminus \{i\}}$ for $\{i\}\subsetneq K$ and  $\{j,\overline{j}\}\not\subset K$ for all $j\in\cv_{2n-2}$.  
			\end{description}
			
			\begin{rema}
				Relations 1, 3 and 4 as descrived above coincide with those appearing in the graph equivariant cohomology of the GKM graph induced from the effective torus action in \cite{ku23}. 
				The only difference is Relation 2.
			\end{rema}
			
			Using a similar proof to that of Theorem~\ref{mainth} (or the main theorem of  \cite{ku23}), we have the following theorem:
			\begin{guess}
				\label{thm-non-effective-action}
				For the GKM graph $\mathcal{GQ}_{2n-2}$, the following isomorphism holds:
				\begin{align*}
					H^{*}(\mathcal{GQ}_{2n-2})\simeq \mathbb{Z}[\mathcal{M}, \mathcal{D}, X]/\mathcal{I}.
				\end{align*}
			\end{guess}
			
			Moreover, the following proposition holds:
			\begin{propo}
				\label{remark_the_correspondence}
				The injective homomorphism $\iota^{*}:H^{*}(\mathcal{GQ}_{2n-1})\to H^{*}(\mathcal{GQ}_{2n-2})$ is induced by the following correspondence of generators: 
				\begin{align*}
					& \iota^{*}:M_{v}\mapsto M'_{v}; \\
					& \iota^{*}:Q\mapsto X; \\
					& \iota^{*}:\Delta_{J}\mapsto X\cdot\Delta'_J.
				\end{align*}
			\end{propo}
			\begin{proof}
				The induced homomorphism $\iota^{*}$ is injective by Corollary~\ref{injectivity_equiv_cohom}.
				
				The first and second correspondences in the statement follow directly from the definitions of generators and $\iota^{*}$.
				The third correspondence is established by comparing the generators from   Theorem~\ref{mainth} and Theorem~\ref{thm-non-effective-action}.
			\end{proof}
			
			\begin{exam}
				In Figure~\ref{comparison_of_two_elements}, we compare two elements: $\Delta'_{1,2,3}\in H^{4}(\mathcal{GQ}_{4})$ and $\Delta_{1,2,3}\in H^{6}(\mathcal{GQ}_{5})$.
				More precisely: 
				The left element $\Delta'_{1,2,3}\in H^{4}(\mathcal{GQ}_{4})$ is defined by: 
				\begin{align*}
					& \Delta'_{1,2,3}(1)=\alpha(14)\alpha(15), \Delta'_{1,2,3}(2)=\alpha(24)\alpha(26),  \Delta'_{1,2,3}(3)=\alpha(35)\alpha(36), \\
					& \Delta'_{1,2,3}(4)=\Delta'_{1,2,3}(5)=\Delta'_{1,2,3}(6)=0.
				\end{align*}
				The right element $\Delta_{1,2,3}\in H^{6}(\mathcal{GQ}_{5})$ is defined by:
				\begin{align*}
					& \Delta_{1,2,3}(1)=\alpha(14)\alpha(15)\alpha(16), \Delta_{1,2,3}(2)=\alpha(24)\alpha(25)\alpha(26), \Delta_{1,2,3}(3)=\alpha(34)\alpha(35)\alpha(36), \\
					& \Delta_{1,2,3}(4)=\Delta_{1,2,3}(5)=\Delta_{1,2,3}(6)=0.
				\end{align*} 
				It is straightforward to verify that $\iota^{*}(\Delta_{1,2,3})=X\Delta'_{1,2,3}$ (see the third correspondence in Proposition~\ref{remark_the_correspondence}).
				
				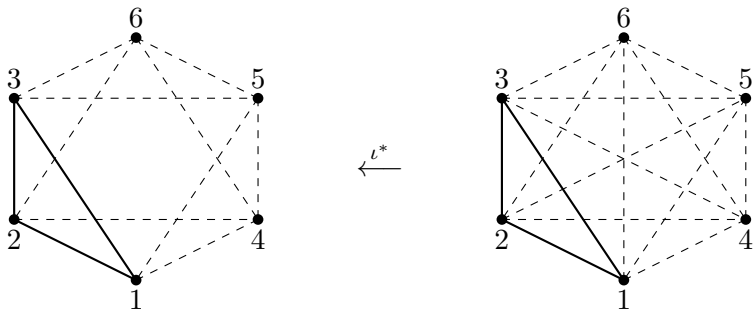
\begin{figure}[h]
					\begin{tikzpicture}
						\begin{scope}[xscale=0.4, yscale=0.4]
							
							\fill(-2,4) circle (5pt);
							\node[above] at (-2,4) {$6$};
							\fill(-6,2) circle (5pt);
							\node[above] at (-6,2) {$3$};
							\fill(-6,-2) circle (5pt);
							\node[below] at (-6,-2) {$2$};
							\fill(-2,-4) circle (5pt);
							\node[below] at (-2,-4) {$1$};
							\fill(2,-2) circle (5pt);
							\node[below] at (2,-2) {$4$};
							\fill(2,2) circle (5pt);
							\node[above] at (2,2) {$5$};
							
							\draw[dashed] (-2,4)--(-6,2);
							\draw[dashed] (-2,4)--(-6,-2);
							\draw[dashed] (-2,4)--(2,-2);
							\draw[dashed] (-2,4)--(2,2);
							\draw[thick] (-6,2)--(-6,-2);
							\draw[thick] (-6,2)--(-2,-4);
							\draw[dashed] (-6,2)--(2,2);
							\draw[dashed] (-6,-2)--(2,-2);
							\draw[thick] (-6,-2)--(-2,-4);
							\draw[dashed] (2,2)--(2,-2);
							\draw[dashed] (2,2)--(-2,-4);
							\draw[dashed] (2,-2)--(-2,-4);
							
							\node at (6,0) {$\stackrel{\iota^{*}}{\longleftarrow}$};
							
							\fill(14,4) circle (5pt);
							\node[above] at (14,4) {$6$};
							\fill(10,2) circle (5pt);
							\node[above] at (10,2) {$3$};
							\fill(10,-2) circle (5pt);
							\node[below] at (10,-2) {$2$};
							\fill(14,-4) circle (5pt);
							\node[below] at (14,-4) {$1$};
							\fill(18,-2) circle (5pt);
							\node[below] at (18,-2) {$4$};
							\fill(18,2) circle (5pt);
							\node[above] at (18,2) {$5$};
							
							\draw[dashed] (14,4)--(10,2);
							\draw[dashed] (14,4)--(10,-2);
							\draw[dashed] (14,4)--(18,-2);
							\draw[dashed] (14,4)--(18,2);
							\draw[thick] (10,2)--(10,-2);
							\draw[thick] (10,2)--(14,-4);
							\draw[dashed] (10,2)--(18,2);
							\draw[dashed] (10,-2)--(18,-2);
							\draw[thick] (10,-2)--(14,-4);
							\draw[dashed] (18,2)--(18,-2);
							\draw[dashed] (18,2)--(14,-4);
							\draw[dashed] (18,-2)--(14,-4);
							\draw[dashed] (14,-4)--(14,4);
							\draw[dashed] (10,-2)--(18,2);
							\draw[dashed] (10,2)--(18,-2);
						\end{scope}
					\end{tikzpicture}
					\caption{$\Delta'_{1,2,3}\in H^{4}(\mathcal{GQ}_{4})$ (left), and $\Delta_{1,2,3}\in H^{6}(\mathcal{GQ}_{5})$ (right).}
					\label{comparison_of_two_elements}
				\end{figure}
			\end{exam}
			
			\begin{rema}
				Given the embedding $\iota:Q_{2n-2}\hookrightarrow Q_{2n-1}$, we obtain the induced homomorphism on ordinary cohomology $\iota^{*}:H^{*}(Q_{2n-1})\to H^{*}(Q_{2n-2})$.
				Note that $\dim Q_{2n-1}=4n-2>4n-4=\dim Q_{2n-2}$; therefore, $H^{4n-2}(Q_{2n-1})\simeq \mathbb{Z}$ and $H^{4n-2}(Q_{2n-2})=0$.
				Thus, $\iota^{*}:H^{*}(Q_{2n-1})\to H^{*}(Q_{2n-2})$ is not injective. 
				However, by Lemma~\ref{key-lemma} and Proposition~\ref{remark_the_correspondence}, the induced homomorphism on equivariant cohomology $\iota^{*}_{T}:H^{*}_{T}(Q_{2n-1})\to H^{*}_{T}(Q_{2n-2})$ is injective.
			\end{rema}
			
			\section{GKM description for non-effective $T^{1}$-actions on $\mathbb{CP}^{1}$}
			\label{sect:8}
			
			In \cite{ku23}, we compute the effective $T^{n}$-ation on $Q_{2n-2}$. 
			According to the main theorem of \cite{ku23}, the generator $X$ that appears in $H^{*}(\mathcal{GQ}_{2n-2})$ in Theorem~\ref{thm-non-effective-action} is not needed (see also Remark~\ref{remark-effective-noneffective}).
			
			In this final section, we observe a similar phenomenon in the equivariant cohomology of non-effective torus actions by comparing the equivariant cohomology of $T^{1}$-actions on $Q_{1}\simeq \mathbb{CP}^{1}\simeq S^{2}$, i.e., we give details of the research announcement \cite[Appendix A]{ku22}.
			For convenience, we denote $T^{1}$ as $T$.
			
			Since the Euler number satisfies $\chi(M)=\chi(M^{T})$ (see \cite[Theorem 41.1]{kawa}), we have $\chi((\mathbb{CP}^{1})^{T})=\chi(S^{2})=2$.
			Therefore, for every non-trivial $T$-action on $\mathbb{CP}^{1}$, there are exactly two fixed points, denote by $(\mathbb{CP}^{1})^{T}=\{p,q\}$.
			
			Using the differentiable slice theorem, the $T$-representations around $T_{p}\mathbb{CP}^{1}$ and $T_{q}\mathbb{CP}^{1}$ are rotations of order $n$, up to sign, for $n\in \mathbb{Z}$.
			More precisely, 
			for every $T$-action on $\mathbb{CP}^{1}$, there exists a non-negative integer $n$ such that the action is weak equivariantly diffeomorphic (i.e., equivariantly diffeomorphic up to an automorphism on $T$) to the following action:
			\begin{align*}
				t\cdot [z_{0}:z_{1}]=[z_{0}:t^{n}z_{1}],
			\end{align*}
			where $t\in T$ and $[z_{0}:z_{1}]\in \mathbb{CP}^1$.
			We denote this action by $\varphi_{n}$ and the equivariant cohomology $H_{T}^{*}(\mathbb{CP}^{1})$ with respect to this action by $H_{\varphi_{n}}^{*}(\mathbb{CP}^{1})$.
			
			\begin{rema}
				Note that the complex quadric $Q_{1}\simeq SO(3)/SO(2)$, with the maximal torus $T~(\subset SO(3))$-action, is equivariantly diffeomorphic to $\mathbb{CP}^{1}$ with the standard effective $T$-action (i.e., $\varphi_{1}$) because there is no non-trivial center in $SO(3)$.
				On the other hand, $\mathbb{CP}^{1}(\simeq Q_{1})$ is also diffeomorphic to $SU(2)/S(U(1)\times U(1))$, but the maximal torus $T ~(\subset SU(2))$-action (i.e., $\varphi_{2}$) has a non-trivial center $\mathbb{Z}_{2}$.
			\end{rema}
			
			For $n=0$, $\varphi_{0}$ represents the trivial $T$-action. 
			Therefore, we have: 
			\begin{align*}
				H_{\varphi_{0}}^{*}(\mathbb{CP}^{1})\simeq H^{*}(\mathbb{CP}^{1})\otimes H^{*}(BT)\simeq \mathbb{Z}[x,\alpha]/\langle x^{2}\rangle,
			\end{align*}
			where $x$ is a generator of $H^{*}(\mathbb{CP}^{1})$ and $\alpha$ is a generator of $H^{*}(BT)$.
			In \cite[Remark 4.5]{kkls}, we also show that: 
			\begin{align*}
				H_{\varphi_{1}}^{*}(\mathbb{CP}^{1})\simeq \mathbb{Z}[\tau_{1},\tau_{2}]/\langle \tau_{1}\tau_{2}\rangle \not\simeq H_{\varphi_{2}}^{*}(\mathbb{CP}^{1})\simeq \mathbb{Z}[u,v]/\langle u^{2}-v^{2} \rangle,
			\end{align*}
			where $\tau_{1}, \tau_{2}, u, v$ are elements of degree two. 
			We will generalize this to arbitrary $\varphi_{n}$ for all $n\ge 0$.
			
			The Mayer-Vietoris exact sequence of the equivariant cohomology is given by: 
			\begin{align*}
				\cdots \longrightarrow H_{\varphi_{n}}^{j}(\mathbb{CP}^{1})\longrightarrow H_{\varphi_{n}}^{j}(U_{0})\oplus H_{\varphi_{n}}^{j}(U_{1})\longrightarrow 
				H_{\varphi_{n}}^{j}(U_{0}\cap U_{1}) \longrightarrow H_{\varphi_{n}}^{j+1}(\mathbb{CP}^{1})\longrightarrow \cdots 
			\end{align*}
			where $U_{0}\simeq \{[z_{0}: 1]\ |\ z_{0}\in \mathbb{C}\}$ is an invariant open neighborhood of the fixed point $[0:1]$, 
			$U_{1}\simeq \{[1: z_{1}]\ |\ z_{1}\in \mathbb{C}\}$ is an invariant open neighborhood of the fixed point $[1:0]$, and 
			$U_{0}\cap U_{1}\simeq \{[z_{0}: z_{1}]\ |\ z_{0}z_{1}\not=0\}\simeq \mathbb{C}^{*}$.
			Here, $U_{i}$ is equivariantly contractible to a point and $U_{0}\cap U_{1}$ is equivariant deformation retractable to the great circle $S^{1}$. 
			Moreover, since $H^{*}(BT)\simeq \mathbb{Z}[\alpha]$ for $\deg \alpha=2$ and $H^{odd}(\mathbb{CP}^{1})=0$, 
			it is well-known that $H^{*}_{\varphi_{n}}(\mathbb{CP}^{1})\simeq H^{*}(BT)\otimes H^{*}(\mathbb{CP}^{1})$ as a module. 
			Hence, 
			this sequence is isomorphic to:
			\begin{align}
				\label{MVexseq}
				0\to 
				H_{T}^{2j-1}(S^{1}) \to H_{T}^{2j}(\mathbb{CP}^{1})\to H^{2j}(BT)\oplus H^{2j}(BT)\to
				H_{T}^{2j}(S^{1}) \to 0.
			\end{align}
			Note that $H_{T}^{*}(S^{1})$ is the equivariant cohomology of the $n$-times rotated action of $T^{1}$ on $S^{1}$.
			Furthermore, the restricted $T^{1}$-action from $\varphi_{n}$ on $S^{1}$ has the kernel $\mathbb{Z}_{n}$ for $n\ge 2$, $\{e\}$ for $n=1$, and $T^{1}$ for $n=0$.
			
			First, consider for $n\ge 1$.
			By \cite[Example 3.41]{hat}, the ring structure of the infinite lens space $B\mathbb{Z}_{n}$ is given by:
			\begin{align*}
				H^{*}(B\mathbb{Z}_{n};\mathbb{Z})\simeq \mathbb{Z}[\alpha]/\langle n\alpha \rangle
			\end{align*} 
			for $\deg \alpha=2$.
			Thus, we have:  
			\begin{align*}
				H_{T}^{  {k}}(S^{1})=H^{  {k}}(ET\times_{T}S^{1})\simeq H^{  {k}}(ET/\mathbb{Z}_{n}) 
				\simeq H^{  {k}}(B\mathbb{Z}_{n})
				\simeq \left\{
				\begin{array}{ll}
					\mathbb{Z} &   {k}=0 \\
					\mathbb{Z}_{n} &   {k}=2j, j>0 \\
					0 &   {k}=2j-1
				\end{array}
				\right.
			\end{align*}
			(Note: $\mathbb{Z}_{1}$ is interpreted as $0$).
			Therefore, by the Mayer-Vietoris sequence \eqref{MVexseq}, we get the following short exact sequence for $j>0$ and $n\ge 1$:
			\begin{align*}
				0\longrightarrow H_{\varphi_{n}}^{2j}(\mathbb{CP}^{1})\longrightarrow \mathbb{Z}\alpha^{j}\oplus \mathbb{Z}\alpha^{j}\longrightarrow \mathbb{Z}_{n}\longrightarrow 0.
			\end{align*}
			On the other hand, for $n=0$, we have: 
			\begin{align*}
				H_{T}^{  {k}}(S^{1})=H^{  {k}}(BT\times S^{1})=H^{  {k}}(BT)\otimes H^{  {k}}(S^{1})
				\simeq \mathbb{Z}\quad \text{for all $  {k}\ge 0$}, 
			\end{align*}
			and therefore, by the Mayer-Vietoris sequence \eqref{MVexseq}, we obtain:
			\begin{align*}
				0\longrightarrow \mathbb{Z}\alpha^{j-1}y\longrightarrow H_{\varphi_{0}}^{2j}(\mathbb{CP}^{1})\longrightarrow \mathbb{Z}\alpha^{j}\oplus \mathbb{Z}\alpha^{j}\longrightarrow \mathbb{Z}\alpha^{j}\longrightarrow 0,
			\end{align*}
			where $y$ is the generator of $H^{  {k}}(S^{1})$.
			Hence, by the definition of the Mayer-Vietoris exact sequence, for all $n\ge 1$, we have: 
			\begin{align*}
				H_{\varphi_{n}}^{*}(\mathbb{CP}^{1})&\simeq \{f\oplus g\in \mathbb{Z}[\alpha]\oplus \mathbb{Z}[\alpha]\ |\ f_{0}=g_{0}, f_{j}-g_{j}\equiv 0\mod n\} \\
				&\simeq \{f\oplus g\in \mathbb{Z}[\alpha]\oplus \mathbb{Z}[\alpha]\ |\ f-g\equiv 0\mod n\alpha\},
			\end{align*}
			where $f=\sum_{i=1}^{k_1}f_i \alpha^i$ and $g=\sum_{i=1}^{k_2}g_i \alpha^i$. 
			For $n=0$, we have: 
			\begin{align}
				H_{\varphi_{0}}^{*}(\mathbb{CP}^{1})&\simeq \{h\oplus f\oplus g\in \mathbb{Z}[\alpha]y\oplus \mathbb{Z}[\alpha]\oplus \mathbb{Z}[\alpha]\ |\ f=g \} \label{trivial-case}
				\\
				&\simeq \mathbb{Z}[x,\alpha]/\langle x^{2} \rangle, \label{trivial-case2}
			\end{align}
			where $x\in H^{2}_{T}(\mathbb{CP}^{1})$ is the image of the generator $y\in H^{1}(S^{1})$ under the connecting homomorpshim $H^{1}(S^{1})\simeq H^{1}_{T}(S^{1})\to H^{2}_{T}(\mathbb{CP}^{1})$. 
			\begin{rema}
				Note that for $n=1$, this description corresponds to the GKM description in the usual sense. Specifically: 
				\begin{align*}
					H_{\varphi_{1}}^{*}(\mathbb{CP}^{1})\simeq \{f\oplus g\in \mathbb{Z}[\alpha]\oplus \mathbb{Z}[\alpha]\ |\ f-g\equiv 0\mod \alpha\}.
				\end{align*}
			\end{rema}
			
			Figure~\ref{GKM_graph_non_effective_2_sphere} illustrates the GKM graph which corresponds to $\varphi_{n}$ for $n\ge 1$.
			Note that the trivial action, $\varphi_{0}$, is not a GKM manifold by definition.
			\begin{figure}[H]
				\begin{tikzpicture}
					\begin{scope}[xshift=150, xscale=1, yscale=1]
						\fill (0,1) circle (1pt); 
						\fill (0,-1) circle (1pt); 
						
						\draw (0,1)--(0,-1);
						\draw[->] (0,1)--(0,0.5);
						\node[right] at (0,0.5) {$n \alpha$};
						\draw[->] (0,-1)--(0,-0.5);
						\node[left] at (0,-0.5) {$-n \alpha$};
						
						\node[right] at (0,1) {$p$};
						\node[left] at (0,-1) {$q$};

					\end{scope}
				\end{tikzpicture}
				\caption{The GKM graph of $\varphi_{n}$ for $n\ge 1$ is depicted, with the fixed points $p=[1:0]$ and $q=[0:1]$.
					Note that the element $\alpha\in \mathfrak{t}^{*}\simeq \mathbb{R}$ can be regarded as the generator of the character lattice $\mathfrak{t}^{*}_{\mathbb{Z}}\simeq \mathbb{Z}\simeq H^{2}(BT^{1})$.}
				\label{GKM_graph_non_effective_2_sphere}
			\end{figure}
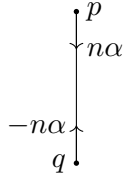
			
			In summary, we have the following GKM description for $\varphi_{n}$.
			\begin{guess}[GKM description for non-effective torus action on $\mathbb{CP}^{1}$]
				\label{GKM description}
				For every non-trivial $T^{1}$-action on $\mathbb{CP}^{1}$, 
				there is the following ring isomorphism:
				\begin{align*}
					H_{\varphi_{n}}^{*}(\mathbb{CP}^{1})\simeq \{f:\{p,q\}\to \mathbb{Z}[\alpha]\ |\ f(p)-f(q)\equiv 0 \mod n\alpha \},
				\end{align*}
				where $\{p,q\}$ denotes the set of fixed points in $\mathbb{CP}^{1}$ and $n\ge 1$.
			\end{guess}
			
			We can also get the ring structure as follows:
			\begin{guess}
				\label{CP1-ring}
				For $n\ge 0$, 
				the ring structure of the equivariant cohomology $H_{\varphi_{n}}^{*}(\mathbb{CP}^{1})$ can be expressed as follows: 
				\begin{align*}
					H_{\varphi_{n}}^{*}(\mathbb{CP}^{1})\simeq \mathbb{Z}[\tau_{p},\tau_{q}, \alpha]/\langle \tau_{p}\tau_{q}, n\alpha-\tau_{p}+\tau_{q} \rangle,
				\end{align*} where 
				\begin{itemize}
					\item $\tau_{p}, \tau_{q}$  are the equivariant Thom classes  associated with the fixed points $p=[1:0]$ and $q=[0:1]$, respectively   {(cf. Figure \ref{generators_of_non_effective_2_sphere})}. 
					\item  $\alpha$ is the pull-back of the generator of $H^{*}(BT)\simeq \mathbb{Z}[\alpha]$.
				\end{itemize}
			\end{guess}
			\begin{proof}
				For $n=0$, the statement follows from \eqref{trivial-case2}.
				Assuming $n\ge 1$, we need to show the isomorphism:
				\begin{align*}
					GKM_{\varphi_{n}}(\mathbb{CP}^{1})&:=\{f:\{p,q\}\to \mathbb{Z}[\alpha]\ |\ f(p)-f(q)\equiv 0 \mod n\alpha \} \\
					& \simeq 
					\mathbb{Z}[\tau_{p},\tau_{q}, \alpha]/\langle \tau_{p}\tau_{q}, n\alpha-\tau_{p}+\tau_{q} \rangle.
				\end{align*}
				In $GKM_{\varphi_{n}}(\mathbb{CP}^{1})$, the equivariant Thom classes are defined as follows: 
				\begin{itemize}
					\item $\tau_{p}(p)=n\alpha$ and $\tau_{p}(q)=0$; 
					\item $\tau_{q}(q)=n\alpha$ and $\tau_{q}(p)=0$; 
					\item $\alpha$ is constant, so $\alpha(p)=\alpha=\alpha(q)$.
				\end{itemize}
				Moreover, it is easy to check that the relations $\tau_{p}\tau_{q}=0$ and $n\alpha=\tau_{p}+\tau_{q}$ hold in $GKM_{\varphi_{n}}(\mathbb{CP}^{1})$.
				Thus, this defines the ring homomorphism 
				\begin{equation}
					\Phi:\mathbb{Z}[\tau_{p},\tau_{q}, \alpha]/\langle \tau_{p}\tau_{q}, n\alpha-\tau_{p}+\tau_{q} \rangle\to GKM_{\varphi_{n}}(\mathbb{CP}^{1}) 
				\end{equation}
				by taking the equivariant Thom classes and $\alpha$ defined as above.
				
				To show that the map $\Phi$ is an isomorphism, we first consider the module structure. 
				Since 
				\begin{align*}
					\mathbb{Z}[\tau_{p},\tau_{q}, \alpha]/\langle \tau_{p}\tau_{q}, n\alpha-\tau_{p}+\tau_{q} \rangle\simeq 
					\mathbb{Z}[\tau_{p}, \alpha]/\langle \tau_{p}(\tau_{p}-n\alpha) \rangle,
				\end{align*}
				we have the following module isomorphism:
				\begin{align}
					\label{module-structure}
					\mathbb{Z}[\tau_{p},\tau_{q}, \alpha]/\langle \tau_{p}\tau_{q}, n\alpha-\tau_{p}+\tau_{q} \rangle\simeq
					\bigoplus_{i=0}^{\infty}\mathbb{Z}\alpha^{i}\oplus 
					\bigoplus_{i=0}^{\infty}\mathbb{Z}\tau_{p}\alpha^{i}.
				\end{align} 
				On the other hand, by the congruence relations,   {every element in $GKM_{\varphi_{n}}(\mathbb{CP}^{1})$ can be written in one of the following forms  $f,g:\{p,q\}\to \mathbb{Z}[\alpha]$,  such that}
				\begin{align*}
					f(p)=\sum_{i=0}^{l}r_{i}   {\alpha^{i}}+\sum_{i=1}^{m}nk_{i}\alpha^{i},\ f(q)=\sum_{i=0}^{l}r_{i}\alpha^{i}
				\end{align*}
				or
				\begin{align*} 
					g(p)=\sum_{i=0}^{l}r_{i}\alpha^{i},\ g(q)=\sum_{i=0}^{l}r_{i}+\sum_{i=1}^{m}nk_{i}\alpha^{i},
				\end{align*}  
				where $r_{i}, k_{i}\in \mathbb{Z}$.
				Namely, we can write using $\tau_{p},\tau_{q}=\tau_{p}-n\alpha$ and $\alpha$ as follows:
				\begin{align*}
					\Phi\left(\sum_{i=0}^{l}r_{i}\alpha^{i}+\tau_{p}\sum_{i=0}^{m-1}k_{i+1}\alpha^{i}\right)=f 
				\end{align*}
				or
				\begin{align*}
					\Phi\left(\sum_{i=0}^{l}r_{i}\alpha^{i}+(\tau_{p}-n\alpha)\sum_{i=0}^{m-1}k_{i+1}\alpha^{i}\right)=g. 
				\end{align*}
				Therefore, $\Phi$ is surjective.
				If $f=0$ (or $g=0$), then $r_{i}=k_{i}=0$ in the above expressions. 
				Therefore, $\Phi$ is injective.
				This establishes the proof. 
			\end{proof}
			
			\begin{figure}[H]
				\begin{tikzpicture}
					\begin{scope}[xshift=150, xscale=1, yscale=1]
						
						\fill (-2,1) circle (1pt); 
						\fill (-2,-1) circle (1pt); 
						\node[right] at (-2,1) {$p$};
						\node[left] at (-2,-1) {$q$};
						\draw (-2,1)--(-2,-1);
						\draw[->] (-2,1)--(-2,0.5);
						\node[right] at (-2,0.5) {$n \alpha$};
						\draw[->] (-2,-1)--(-2,-0.5);
						\node[left] at (-2,-0.5) {$0$};
						
						\fill (0,1) circle (1pt); 
						\fill (0,-1) circle (1pt); 
						\node[right] at (0,1) {$p$};
						\node[left] at (0,-1) {$q$};
						\draw (0,1)--(0,-1);
						\draw[->] (0,1)--(0,0.5);
						\node[right] at (0,0.5) {$0$};
						\draw[->] (0,-1)--(0,-0.5);
						\node[left] at (0,-0.5) {$-n\alpha$};
						
						\fill (2,1) circle (1pt); 
						\fill (2,-1) circle (1pt); 
						\node[right] at (2,1) {$p$};
						\node[left] at (2,-1) {$q$};
						\draw (2,1)--(2,-1);
						\draw[->] (2,1)--(2,0.5);
						\node[right] at (2,0.5) {$\alpha$};
						\draw[->] (2,-1)--(2,-0.5);
						\node[left] at (2,-0.5) {$\alpha$};
						
					\end{scope}
				\end{tikzpicture}
				\caption{Visualizing the generators $\tau_{p}, \tau_{q}$ and $\alpha$ (from left) of $H_{\varphi_{n}}^{*}(\mathbb{CP}^{1})$ for $n\ge 1$.}
				\label{generators_of_non_effective_2_sphere}
			\end{figure}
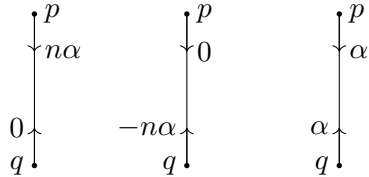
			
			\begin{rema}
				In \cite{zo}, Zollar also studies the GKM theory of non-effective torus actions in a more general setting.
			\end{rema}
			
			\section*{Acknowledgment}
		S. Kuroki was supported by JSPS KAKENHI Grant Number 21K03262. B. Paul acknowledges the Chennai Mathematical Institute and the Infosys Foundation for their financial support. The authors are grateful to the anonymous referees for their careful reading of the manuscript and for their valuable comments and suggestions. We also sincerely thank Prof. Suyoung Choi for pointing out a mistake in an earlier version of the manuscript and for his helpful comments.


\begin{thebibliography}{99}
				\bibitem{bfr09} A.~Bahri, M.~Franz, and N.~Ray, 
				{\it The equivariant cohomology ring of weighted projective space}, 
				Math.~Proc.~Cambridge Philos.~Soc.~{\bf 146} no. 2, 395--405, 2009.
				
				\bibitem{dks22}	A. Darby, S. Kuroki and J. Song, {\it Equivariant cohomology of torus orbifolds}, Canadian J. of Math., {\bf 74},	Issue 2, 299--328, 2022.
				
				\bibitem{eh13} D.~Eisenbud and J.~Harris,
				{\it 3264 and All That: Intersection Theory in Algebraic Geometry}, 
				Cambridge University Press, Cambridge, 2013.
				
				\bibitem{EKM08}
				R.~Elman, N.~Karpenko and A.~Merkurjev,
				{\it The algebraic and geometric theory of quadratic forms}, 
				Amer.~Math.~Soc.~Colloquium Publ., \textbf{56}. 
				American Mathematical Society, Providence, RI, 2008.
				
				
				\bibitem{f10}
				M.~Franz, 
				{\it Describing toric varieties and their equivariant cohomology}, 
				Colloq. Math. {\bf 121}, no. 1, 1--16, 2010.
				
				\bibitem{fp07}	M. Franz and V. Puppe, {\it Exact cohomology sequences with integral coefficients for torus actions}, Transform. Groups {\bf 12}, no. 1, 65--76, 2007.
				
				\bibitem{fp11} M. Franz and V. Puppe, {\it  Exact sequences for equivariantly formal spaces}, C.~R.~Math.~Acad.~Sci.~Soc.~R.~Can.~{\bf 33}, 1--10, 2011.
				
				
				\bibitem{gkm} M. Goresky, R. Kottwitz, and R. MacPherson. {\it Equivariant cohomology, Koszul duality, and the localization theorem},  Invent. Math. {\bf 131}(1): 25--83, 1998.
				
				
				
				\bibitem{ghz06} V. Guillemin, T. Holm, and C. Zara, {\it A GKM description of the equivariant cohomology ring of a homogeneous space} { J. Algebraic Combin.} {\bf 23}(1): 21--41, 2006.
				
				
				\bibitem{gz01} V. Guillemin and C. Zara. {\it 1-skeleta, Betti numbers, and equivariant cohomology}, { Duke Math. J.} {\bf 107}(2): 283--349, 2001. 
				
				\bibitem{hat} A. Hatcher, {\it Algebraic Topology}, (2002) Cambridge University Press, Cambridge.
				
				\bibitem{kkls}	S. Kaji, S. Kuroki, E. Lee and D. Y. Suh, {\it Flag Bott manifolds of general Lie	type and their equivariant cohomology rings}, Homology Homotopy Appl. {\bf 22}(1):	375--390, 2020.
				
				\bibitem{kawa} K.~Kawakubo,
				The theory of transformation groups, Oxford Univ. Press, London, 1991.
				
				
				\bibitem{ku22} S. Kuroki {\it Equivariant cohomology of complex quadrics from a combinatorial point of view}, RIMS Kokyuroku 2231, 85--99, 2022.
				
				\bibitem{ku23} S. Kuroki, {\it Equivariant cohomology of even-dimensional complex quadrics from a combinatorial point of view}, Osaka J. Math. {\bf 62}, No 4, (2025) pp. 539--563. 
				
				\bibitem{La74} H.~Lai, \emph{On the topology of the even-dimensional complex quadrics},
				Proc.~Amer.~Math.~Soc.~\textbf{46}, No.~3, (1974), 419--425.
				
				
				\bibitem{mmp07} H. Maeda, M. Masuda, and T. Panov, {\it Torus graphs and simplicial posets}, Adv. Math., {\bf 212}(2): 458-–483, 2007.
				
				\bibitem{b24}  B. Paul, {\it Equivariant K-theory of even-dimensional complex quadrics}, Topology Appl. {\bf 368} (2025), Paper No. 109376. 
				
				\bibitem{Se06} J.~Seade, 
				{\it On the Topology of Isolated Singularities in Analytic Spaces}, 
				Progress in Mathematics, 241. Birkh${\rm \ddot{a}}$user Verlag, Basel, 2006.
				
				\bibitem{zo} 
				L.~Zoller, {\it On integral Chang-Skjelbred computations with disconnected isotropy groups}, Transformation Groups (2025): 1–17.
				
				
				
			\end{thebibliography}
		\end{document}